\documentclass[hidelinks,onefignum,onetabnum]{siamart251216}

\usepackage{amsfonts}
\usepackage{graphicx}
\usepackage{color}

\newsiamremark{remark}{Remark}
\newsiamremark{hypothesis}{Hypothesis}
\crefname{hypothesis}{Hypothesis}{Hypotheses}
\newsiamthm{claim}{Claim}

\headers{Fast Gradient Algorithm for Nonconvex Optimization}{L. T. Nguyen, A. Eberhard, X. Yu, and C. Li}
\title{Fast Gradient Algorithm with Dry-like Friction and Nonmonotone Line Search for Nonconvex Optimization Problems\thanks{Submitted to the editors DATE.
\funding{This work was supported by the Australian Research Council (ARC) through Discovery Program under Grant DP200101197.}}}

\author{Lien T. Nguyen\thanks{School of Science, RMIT University, Melbourne, VIC 3000, Australia
  (\email{nguyenthuylien228@gmail.com},  \email{andy.eberhard@rmit.edu.au}).}
\and Andrew Eberhard\footnotemark[2], 
\and Xinghuo Yu\thanks{School of Engineering, RMIT University, Melbourne, VIC 3000, Australia
  (\email{xinghuo.yu@rmit.edu.au}).}
\and Chaojie Li\thanks{School of Electrical Engineering and Telecommunications, University of New South Wales, NSW 2033, Australia
(\email{chaojie.li@unsw.edu.au}).}}

\ifpdf
\hypersetup{
  pdftitle={Fast Gradient Algorithm with Dry-like Friction and Nonmonotone Line Search for Nonconvex Optimization Problems},
  pdfauthor={L. T. Nguyen, A. Eberhard, X. Yu, and C. Li}
}
\fi

\usepackage{amsmath, amssymb}
\usepackage[breakable]{tcolorbox}
\usepackage[caption=false]{subfig}
\usepackage{multirow}

\newsiamremark{example}{Example}
\newsiamremark{algo}{Algorithm}

\renewcommand\theenumi{(\roman{enumi})}
\renewcommand\theenumii{(\alph{enumii})}
\renewcommand{\labelenumi}{\rm (\roman{enumi})}

\usepackage[shortlabels]{enumitem}

\makeatletter
\renewcommand{\normalsize}{%
   \@setfontsize\normalsize\@xpt\@xiipt
   \abovedisplayskip 5\p@ \@plus2\p@ \@minus5\p@
   \abovedisplayshortskip \z@ \@plus3\p@ 
   \belowdisplayshortskip 3\p@ \@plus3\p@ \@minus3\p@ 
   \belowdisplayskip \abovedisplayskip
   \let\@listi\@listI}
\makeatother   

\newcommand{\argmax}{\ensuremath{\operatorname*{argmax}}}
\newcommand{\argmin}{\ensuremath{\operatorname*{argmin}}}

\newcommand{\dom}{\ensuremath{\operatorname{dom}}}

\newcommand{\prox}{\ensuremath{\operatorname{Prox}}}

\newcommand{\dist}{\ensuremath{\operatorname{dist}}}

\newcommand{\R}{\mathbb{R}}
\newcommand{\N}{\mathbb{N}}

\newcommand{\epi}{\operatorname{epi}}
\newcommand{\inte}{\operatorname{int}}
\newcommand{\Id}{\operatorname{Id}}

\newcounter{step}
\newcommand\step[1]{%
	\refstepcounter{step}	
	\vskip 0.25\baselineskip
	\ifx\hfuzz#1\hfuzz
		\item[~\(\triangleright\)~\textbf{Step~\arabic{step}.}]
	\else
		\item[~\(\triangleright\)~\textbf{Step~\arabic{step}}] (\texttt{#1})\textbf{.}%
	\fi
}

\allowdisplaybreaks

\begin{document}

\maketitle

\begin{abstract}
In this paper, we propose a fast gradient algorithm for the problem of minimizing a differentiable (possibly nonconvex) function in Hilbert spaces. We first extend the dry friction property for convex functions to what we call the \emph{dry-like friction property} in a nonconvex setting, and then employ a line search technique to adaptively update parameters at each iteration. Depending on the choice of parameters, the proposed algorithm exhibits subsequential convergence to a critical point or full sequential convergence to an ``approximate'' critical point of the objective function. We also establish the full sequential convergence to a critical point under the Kurdyka--{\L}ojasiewicz (KL) property of a merit function. Thanks to the parameters' flexibility, our algorithm can reduce to a number of existing inertial gradient algorithms with Hessian damping and dry friction. By exploiting variational properties of the Moreau envelope, the proposed algorithm is adapted to address weakly convex nonsmooth optimization problems. In particular, we extend the result on KL exponent for the Moreau envelope of a convex KL function to a broad class of KL functions that are not necessarily convex nor continuous. Simulation results illustrate the efficiency of our algorithm and demonstrate the potential advantages of combining dry-like friction with extrapolation and line search techniques.
\end{abstract}



\section{Introduction}

Throughout this paper, $\mathcal{H}$ is a real Hilbert space equipped with the scalar product $\langle \cdot, \cdot\rangle$ and the associated norm $\|\cdot\|$. We consider the nonconvex optimization problem
\begin{equation}\label{eq:pb}
\min_{x \in \mathcal{H}} f(x),
\end{equation}
where $f\colon \mathcal{H} \to \mathbb{R}$ is a differentiable function. Using dynamical systems and their discretization to solve this problem has attracted much attention from very early days. One of the most popular algorithms is the gradient descent (GD). However, it is known that the GD method converges slowly
and can easily be trapped into local optima. Thus, there has been many studies on the theory and
practice of both accelerated first-order and second-order schemes, for example \cite{AA22,AAC06, ABC21, BBJ15, NYEL22, SBC14}. Among them are the algorithms that use dry friction, viscous friction, and a geometric damping driven by the second-order information of the objective. In \cite{Pol64}, Polyak introduced the \emph{heavy-ball method} with a fixed viscous damping coefficient $\lambda > 0$, namely 
\begin{equation*}
\ddot{x}(t) + \lambda \dot{x}(t) + \nabla f(x(t)) = 0.
\end{equation*}
When $f$ is convex, this method has the asymptotic convergence rate of $O(\frac{1}{t})$, which is however not better than the GD method. The next important step was taken by Su, Boyd, and Candes \cite{SBC14} with the introduction of an asymptotic vanishing damping coefficient $\frac{\lambda}{t}$, that is
\begin{equation*}
\ddot{x}(t) + \frac{\lambda}{t} \dot{x}(t) + \nabla f(x(t)) = 0.
\end{equation*}
For $\lambda \geq 3$, the trajectories of this method satisfy $f(x(t)) -\inf_{\mathcal{H}} f= O(\frac{1}{t^2})$.

Dry friction which is an important subject in mechanics, was first introduced by Adly, Attouch, and Cabot \cite{AAC06} for the continuous dynamics
\begin{equation*}
\ddot{x}(t) + \partial \phi(\dot{x}(t))  + \nabla f(x(t)) \ni 0.
\end{equation*}
While viscous dampings produce asymptotically many small oscillations, dry friction can stabilize mechanical systems in finite time. Assuming that the potential friction function $\phi$ has a sharp minimum at the origin (dry friction), they showed that the trajectories converge to equilibria in finite time. Corresponding results for partial differential equations were obtained by Amann and Diaz in \cite{AD03}.

In \cite{AA22}, Adly and Attouch analyzed the convergence properties of several algorithms obtained by temporal discretization of the differential inclusion
\begin{equation*}
\ddot{x}(t) + \gamma(t) \dot{x}(t) + \partial \phi (\dot{x}(t)) + \nabla f(x(t)) \ni 0.
\end{equation*}
Their main results concern the \emph{inertial gradient proximal algorithm with dry friction} (IPGDF)		
\begin{equation*}
x_{n+1} =x_n + h \prox_{\frac{h}{1 +h\gamma}\phi}\left(\frac{1}{h(1+h\gamma)} (x_{n} -x_{n-1}) -\frac{h}{1+h\gamma} \nabla f(x_n))\right).
\end{equation*}
The combination of viscous friction with dry friction and Hessian-driven damping has been considered by Adly and Attouch in \cite{AA20, AA21} as
\begin{equation*}
\ddot{x}(t) + \gamma \dot{x}(t) +  \partial \phi (\dot{x}(t)) + \beta \nabla^2 f(x(t)) \dot{x}(t) + \nabla f(x(t)) \ni 0
\end{equation*}	
and the discrete-time algorithm 
\begin{align*}
x_{n+1}=x_n+ h u_n, \text{~where}
\end{align*}
$u_n = \prox_{\frac{h}{1 +h\gamma} \phi}\left(\tfrac{1}{h(1+h\gamma)} (x_{n} -x_{n-1})-\tfrac{\beta}{1+ h\gamma}(\nabla f(x_n) -\nabla f(x_{n-1})) -\tfrac{h}{1+h\gamma} \nabla f(x_n))\right).$
The Hessian-driven damping has a natural connection with the strong damping property in mechanics and physics. It helps control and attenuate the oscillation effects that occur naturally with inertial systems. But each sequence $(x_n)_{n \in \N}$ generated by the algorithms in \cite{AA20, AA21, AA22, AAC06} only converges to an ``approximate'' critical point $x_{\infty}$ of $f$ in the sense that $-\nabla f(x_{\infty}) \in \partial \phi(0)$.

In \cite{ABC21}, Attouch, Bo\c{t}, and Csetnek introduced the differential inclusion
\begin{equation}\label{eq:ABC}
\ddot{x}(t) + \gamma \dot{x}(t) +  \partial \phi (\dot{x}(t) +\beta \nabla f(x)) + \beta\nabla^2 f(x(t)) \dot{x}(t) + \nabla f(x(t)) \ni 0,
\end{equation}	
where the damping term $\partial \phi (\dot{x}(t) +\beta \nabla f(x))$ involves both the velocity vector and the gradient of the potential function $f$. An advantage of considering the dry friction term in this new form is that the iterates generated by that algorithm converge towards a critical point of $f$ (a minimizer in the case where $f$ is convex).
In \cite{AAL21}, Adly, Attouch, and Le showed various discretizations of the dynamic \eqref{eq:ABC}, where the main concern is the \emph{inertial proximal algorithm with Hessian damping and dry friction} (IPAHDD-C1) 
\begin{equation*}
\begin{cases}
z_n &= \frac{1}{h}(x_n -x_{n-1}) + \beta \nabla f(x_{n-1}) \\
x_{n+1}\!\!\! &= x_n - \beta h \nabla f(x_n) + h \prox_{\frac{h}{1 +h\gamma} \phi} \left( \frac{1}{1 +\gamma h} z_n + \frac{(\gamma \beta -1)h}{1 +\gamma h} \nabla f(x_n)\right).
\end{cases}
\end{equation*}
Despite their good convergence properties, the algorithms based on the dry friction damping are not as fast as Nesterov's accelerated gradient algorithm \cite{Nes83}. Although the authors also used the combination of dry friction and Hessian-driven damping with inertial techniques, their methods follow the GD method after a finite number of steps. Moreover, the choices of parameters are relatively restrictive and may not work well for achieving acceleration. To get more flexibilities in the choices of the parameters in solving optimization problems, the line search technique is an efficient strategy. A traditional monotone line search scheme is given under the Wolfe condition or the Armijo condition, see \cite{Arm66, Wol69}. The Armijo one tries to get a small value that descends the function and the Wolfe one tries to prevent the selection of excessively small values, which can lead to slow convergence. A bisection strategy is used to get both satisfied simultaneously. In \cite{GLL86}, Grippo, Lampariello, and Lucidi originally introduced a nonmonotone line search for Newton's methods. Subsequent works, including \cite{Dai02, ZH04} have also employed nonmonotone line search techniques. Nevertheless, the authors in \cite{Dai02, ZH04} achieved $R$-linear convergence for their nonmonotone line search algorithms only when the objective function is strongly convex. Recognizing the potential of nonmonotone line search schemes to enlarge the choices of related inertial parameters and improve performance, researchers extended this technique to address nonconvex nonsmooth problems in \cite{GZLHY13, LL15, WNF09, Yan21}.

In this work, we first extend the dry friction property for convex functions to the dry-like friction property for a broad class of possibly nonconvex functions. Using functions satisfying this dry-like friction property, we propose a \emph{fast gradient algorithm with dry-like friction and nonmonotone line search} (AGDNL). It is based on a combination of Hessian damping, dry-like friction, Nesterov's accelerated technique, and nonmonotone line search. Thanks to the line search technique, the parameters can be chosen more flexibly and be updated adaptively if a certain line search criterion is not satisfied. Depending on the choice of parameters, the sequence generated by the proposed algorithm is bounded (with all strong cluster points being critical points) or strongly convergent to an ``approximate'' critical point of the objective function. Moreover, with appropriate choices of parameters, our algorithm reduces to several existing algorithms. 
To the best of our knowledge, this is the first algorithm using the combination of dry friction property and line search technique. On the other hand, the proposed algorithm also includes a sequence of perturbations and errors which is only required to be bounded instead of converging to zero. Under the KL property of a suitable merit function, we establish the convergence of the full sequence generated by our algorithm. We also adapt the proposed algorithm to address problem \eqref{eq:pb} in the case where $f\colon \mathcal{H}\to (-\infty,+\infty]$ is a weakly convex function. In addition, we extend the result on the KL exponent of the Moreau envelope of a convex KL function in finite-dimensional settings to a broad class of possibly nonconvex KL functions in Hilbert spaces. 

The remainder of the paper is organized as follows. We start Section~\ref{s:DLF} by recalling some necessary mathematical notions and then introduce the dry-like friction property for possibly nonconvex functions. In Section~\ref{s:main}, our AGDNL algorithm is developed together with its convergence properties in different scenarios. Section~\ref{s:comparison} discusses some special cases of parameters such that the line search criterion automatically holds and the proposed algorithm reduces to several existing algorithms. Based on the properties of the Moreau envelope, these results are extended to nonconvex nonsmooth functions in Section~\ref{s:nonsmooth}. Numerical simulations are conducted in Section~\ref{s:numerical simulation} to confirm the superior performance of our proposed algorithm. Finally, Section~\ref{s:conclusion} provides some concluding remarks.

\section{Preliminaries}

We denote the set of nonnegative integers by $\mathbb{N}$, the set of real numbers by $\mathbb{R}$, the set
of nonnegative real numbers by $\mathbb{R}_+$, and the set of the positive real numbers by $\mathbb{R}_{++}$. For $x\in \mathcal{H}$ and $\delta \in \mathbb{R}_+$, the closed ball in $\mathcal{H}$ centered at $x$ with radius $\delta$ is $B(x; \delta) = \{y \in \mathcal{H}: \|x - y\| \leq \delta\}$. For a subset $C$ of $\mathcal{H}$, its interior is $\inte C =\{x\in C: \exists \delta \in \mathbb{R}_{++},\ B(x; \delta) \subseteq C\}$.

Given $f\colon \mathcal{H} \to (-\infty, +\infty]$, its \emph{domain} is denoted by $\dom f =\{x \in \mathcal{H} \colon f(x)  < +\infty\}$ and its \emph{epigraph} by $\epi f =\{(x,\xi) \colon f(x) \leq \xi\}$. We say that $f$ is \emph{proper} if $\dom f \neq \varnothing$, \emph{lower semicontinuous} if $\epi f$ is a closed set, and \emph{convex} if $\epi f$ is a convex set. The function $f$ is said to be \emph{$\alpha$-weakly convex} ($\alpha \in \mathbb{R}_+$) if $f + \frac{\alpha}{2}\| \cdot\|^2$ is convex.

Let $f\colon \mathcal{H} \to (-\infty, +\infty]$ be proper and let $x \in \dom f$. The \emph{Fr\'echet subdifferential} of $f$ at $x$ is defined as
\begin{equation*}
\hat{\partial} f(x) = \left \{u \in \mathcal{H} \colon \liminf_{y \to x}\frac{f(y) -f(x) -\langle u, y-x \rangle}{\|y-x\|} \geq 0\right \}
\end{equation*}
and the \emph{limiting subdifferential} of $f$ at $x$ is defined as
\begin{equation*}
\partial f(x) = \{u \in \mathcal{H} \colon \exists x_n \to x, f(x_n) \to f(x) \text{~and~} \exists u_n \in \hat{\partial} f(x_n), u_n \to u \text{~as~} n \to +\infty\}.
\end{equation*}
For $x \notin  \dom f$, one takes $\hat{\partial} f(x)=\varnothing$ and $\partial f(x) = \varnothing$. The domain of $\partial f$ is $\dom \partial f = \{ x \in \mathcal{H}: \partial f(x) \neq \varnothing\}$. In the case when $f$ is convex, both Fr\'echet and limiting subdifferentials coincide with the \emph{convex subdifferential}, that is,
 $\partial f(x) =\{u \in \mathcal{H}: \forall y \in \mathcal{H}, \langle u, y-x\rangle \leq f(y) -f(x)\}.$
Let $\xi \in \R_{++}$. The \emph{Moreau envelope} of $f$ is the mapping $f_{\xi} \colon \mathcal{H} \to \left(-\infty, +\infty\right]$ given by
\begin{equation*}
f_{\xi}(x) = \inf_{y \in \mathcal{H}} \left(f(y) + \frac{1}{2\xi}\|x-y\|^2\right)
\end{equation*} 
and the \emph{proximal operator} of $f$ is the set-valued mapping $\prox_{\xi f} \colon  \mathcal{H} \rightrightarrows \mathcal{H}$ given by
\begin{equation*}
\prox_{\xi f}(x) =\argmin_{y \in \mathcal{H}}  \left(f(y) + \frac{1}{2\xi}\|x-y\|^2\right).
\end{equation*}
In a common abuse of notation, we consider $\prox_{\xi f}(x)$ as a point if it is a singleton.

\subsection*{Dry-like friction property}
\label{s:DLF}
We recall from \cite{AA20, AA21, AA22} that a function $\phi \colon \mathcal{H} \to \R_+$ is said to satisfy the \emph{dry friction property} if $\phi$ is convex and continuous, $\min_{x \in \mathcal{H}} \phi(x) = \phi(0) =0$, and there exists $r \in \mathbb{R}_{++}$ such that, for all $x\in \mathcal{H}$, $\phi(x) \geq r\|x\|$. We now extend this concept to nonconvex settings.

\begin{definition}[Dry-like friction property]
\label{d:DLF}
We say that a proper function $\phi\colon \mathcal{H} \to \left(-\infty, +\infty\right]$ satisfies the \emph{dry-like friction property with constant $r \in \mathbb{R}_{++}$} if for all $x \in \dom \partial \phi$, $\inf_{u \in \partial \phi(x)} \langle u, x\rangle \geq r \|x\|$. The function $\phi$ is said to satisfy the \emph{dry-like friction property} or called a \emph{dry-like friction function} if it satisfies dry-like friction property with some constant $r$.
\end{definition}

\begin{remark}
\label{r:DLF}
We have the following comments on the dry-like friction property.

\begin{enumerate}
\item\label{r:DLF_weaklycvx}
Let $\phi \colon \mathcal{H} \to \left(-\infty, +\infty\right]$ be a proper $\alpha$-weakly convex function and suppose that there exists $r \in \mathbb{R}_{++}$ such that, for all $x\in \mathcal{H}$, 
\begin{equation*}
\phi(x) \geq \frac{\alpha }{2}\|x\|^2 + r\|x\| + \phi(0).    
\end{equation*}
Then $\phi$ satisfies the dry-like friction property with constant $r$. Indeed, take arbitrary $x \in \dom \partial \phi$ and $u \in \partial \phi(x)$. By the $\alpha$-weak convexity, $\langle u, 0-x \rangle \leq  \phi(0) - \phi(x) +\frac{\alpha}{2}\|x\|^2$, and so
\begin{equation*}
\langle u, x \rangle \geq \phi(x) - \phi(0) - \frac{\alpha}{2}\|x\|^2\geq \frac{\alpha}{2}\|x\|^2 +r\|x\|+\phi(0)- \phi(0) -\frac{\alpha}{2}\|x\|^2= r\|x\|.
\end{equation*}
Therefore, $\inf_{u \in \partial \phi(x)}\langle u, x \rangle \geq r\|x\|$.
\item
If $\phi$ satisfies the dry friction property, then, by \ref{r:DLF_weaklycvx}, it satisfies the dry-like friction property. 
\item\label{r:DLF_cvx} 
Let $r \in \mathbb{R}_{++}$. According to \ref{r:DLF_weaklycvx}, $\phi =r\|\cdot\|$ satisfies the dry-like friction property with constant $r$. In this case, we also note that, for all $\tau \in \mathbb{R}_{++}$,
\begin{equation*}
\prox_{\tau\phi}(x) = \left(1- \frac{r\tau}{\max\{r\tau, \|x\|\}}\right) x = \begin{cases}
0 & \text{~if~} \|x\| \leq r\tau,\\
(\|x\| -r\tau)\frac{x}{\|x\|} &\text{~if~} \|x\| \geq r\tau.
\end{cases}
\end{equation*}
\item\label{r:DLF_weaklycvxfun}
Let $\alpha, r \in \mathbb{R}_{++}$ and $\phi: \mathcal{H} \to \left(-\infty, +\infty\right]$ be given by $\phi(x) = \max\{\frac{\alpha}{2}\|x\|^2 +r\|x\|, -\frac{\alpha}{2}\|x\|^2+2r\|x\|\}$. Then $\phi(x) + \frac{\alpha}{2}\|x\|^2 = \max\{\alpha\|x\|^2+r\|x\|, 2r\|x\|\}$ is a proper convex function, and so $\phi$ is proper and $\alpha$-weakly convex. Moreover, we have from the definition of $\phi $ that $\phi(0) =0$ and, for all $x\in \mathcal{H}$, $\phi(x) \geq \frac{\alpha }{2}\|x\|^2 + r\|x\| + \phi(0)$. By \ref{r:DLF_weaklycvx}, $\phi$ satisfies the dry-like friction property with constant $r$. We also have the closed form solution for the proximal operator of $\phi$. Specifically, for all $\tau \in (0,\frac{1}{\alpha})$, 
\begin{equation*}
\prox_{\tau\phi}(x) = \begin{cases}
0 & \text{if~} \|x\| \leq 2r\tau,\\
\frac{\|x\| -2r\tau}{1 -\alpha \tau}\frac{x}{\|x\|} &\text{if~} 2 r\tau < \|x\| \leq r\tau + \frac{r}{\alpha},\\
\frac{r}{\alpha} \frac{x}{\|x\|} & \text{if~} r\tau + \frac{r}{\alpha} < \|x\| \leq 2 r\tau + \frac{r}{\alpha}, \\
\frac{\|x\| -r\tau}{ 1+\alpha \tau }\frac{x}{\|x\|} & \text{if~} \|x\| > 2 r\tau +\frac{r}{\alpha}.
\end{cases}
\end{equation*} 
The proof of this is given in Appendix~\ref{s:appendix}.

\item \label{r:DLF_stribeckfun}
Inspired by mechanical models involving Stribeck friction,
we consider the function $\phi: \mathcal{H} \to (-\infty, +\infty]$ given by
\begin{equation*}
\phi(x) = \frac{1}{2}k_v\|x\|^2 + F_c\|x\| - v_s(F_s - F_c) \exp\left(-\frac{\|x\|}{v_s}\right),
\end{equation*}
where $k_v \in \R_+$, $v_s \in \R_{++}$, $F_c \in \R_{++}$, and $F_s \in [F_c, +\infty)$. Then 
\begin{equation}
\partial \phi(x)= \begin{cases}
k_v x + F_c \frac{x}{\|x\|} + (F_s -F_c)\exp(-\frac{\|x\|}{v_s})\frac{x}{\|x\|} &\text{~if~} x \neq 0,\notag\\
B(0; F_s) &\text{~if~} x=0
\end{cases}
\end{equation}
and we see that $\phi$ satisfies the dry-like friction property with constant $F_c$.
\end{enumerate}
\end{remark}

\begin{lemma}\label{l:dry}
Suppose that $\phi\colon \mathcal{H}\to \left(-\infty, +\infty\right]$ satisfies the dry-like friction property with constant $r$. Let $y \in \mathcal{H}$ and $w \in \prox_{\tau \phi}(y)$. Then $\langle w-y, w\rangle  + r\tau \|w\| \leq 0.$
\end{lemma}
\begin{proof}
Since $w \in \prox_{\tau \phi}(y) = \argmin_{x \in \mathcal{H}} \left(\phi(x) +\frac{1}{2 \tau}\|x-y\|^2\right)$, we have $0 \in \partial \phi(w) + \frac{1}{\tau}(w-y)$, which yields $\frac{1}{\tau}(y-w) \in \partial \phi(w)$, and so $w \in \dom \partial \phi$. By the dry-like friction property of $\phi$, $\frac{1}{\tau}\langle y -w, w\rangle \geq r\|w\|$, which completes the proof.
\end{proof}

\section{Combination of dry-like friction and line search}
\label{s:main}

In this section, we assume that $f: \mathcal{H} \to \R$ is a differentiable function whose gradient is $L$-Lipschitz continuous. Inspired by the works using dry friction in \cite{AA20,AA21, AA22, AAL21, ABC21} and the works using nonmonotone line search in \cite{Dai02,  GZLHY13,GLL86, WNF09,  Yan21, ZH04} to enlarge the choices of the parameters, we propose the AGDNL for solving problem \eqref{eq:pb} with the potential function $V \colon \mathcal{H} \times \mathcal{H} \times \mathcal{H} \times \R_+ \times \R_+ \to \R $ defined by
\begin{align}\label{eq:V}
V(x,y,w,\rho,\sigma)= f(x) +\rho \|x -y\|^2 +\sigma \|w\|^2.
\end{align}
The complete framework of the proposed algorithm is presented in Algorithm~\ref{algo:AGDNL}. In this algorithm, the Lipschitz constant is not required to be known and only approximate Lipschitz constants are needed. This is very useful because, for large-scale problems, the Lipschitz constant is not always known nor always easily computable.

\begin{tcolorbox}[
	left=0pt,right=0pt,top=0pt,bottom=0pt,
	colback=blue!10!white, colframe=blue!50!white,
  	boxrule=0.2pt]
\begin{algo}[AGDNL]
\label{algo:AGDNL}
\step{}\label{step:init}
Let $\phi\colon \mathcal{H}\to \left(-\infty, +\infty\right]$ satisfy the dry-like friction property with constant $r$. 
Let $u_{-1}=x_{-1} =x_0 \in \mathcal{H}$. 
Let $\mu \in \R_{++}$, $\bar\alpha =\alpha_{-1} \in (0, \frac{1}{\mu})$,  $\bar\beta =\beta_{-1} \in \R_{+}$, and $\bar\gamma =\gamma_{-1} \in \R_{+}$. 
Let $L_{\min} \in (0, L]$, and $L_{\max} \in [L, +\infty)$. Let $\lambda_{\min} \in [0,\frac{1}{L_{\max}})$, $\lambda_{\max} \in [\lambda_{\min}, +\infty)$, $\kappa_{\min} \in (0, \bar\kappa]$, where $\bar\kappa = \frac{2(1-\mu \bar\alpha)}{\mu L_{\max}} -\frac{2(1-\mu \bar\alpha)\lambda_{\min}}{\mu}$, and $\kappa_{\max} \in [\kappa_{\min}, +\infty)$. Set 
\begin{equation*}
\bar c =\frac{2(1-\mu \bar \alpha) - \mu^2 L_{\max} \underline{\nu}}{2 \mu^2 \underline{\nu}} 
\text{~~and~~} \bar d =\frac{(1-\mu \bar\alpha)\lambda_{\min}^2 + \lambda_{\min} \mu \kappa_{\min}}{\mu^2 \underline{\nu}},    
\end{equation*}
where $\underline{\nu}=\frac{2 \lambda_{\min}(1 -\mu \bar\alpha)}{\mu^2}+ \frac{\kappa_{\min}}{\mu}$. Let $c \in [0, \bar c ]$ and $ d \in [0, \bar d]$ such that $d >0$ if $\lambda_{\min} >0$ and $c=0$. Let $L_{-1} \in [L_{\min}, L_{\max}]$, $\lambda_{-1} \in [\lambda_{\min}, \lambda_{\max}]$, $\kappa_{-1} \in [\kappa_{\min}, \kappa_{\max}]$, and set $\rho_{-1} = \frac{1}{\mu^2}(1 -\mu \bar\alpha)\frac{\gamma_{-1}}{\nu_{-1}} + \frac{L_{-1}\beta_{-1} +\gamma_{-1}}{2}$ and $\sigma_{-1} =\frac{\mu \alpha_{-1}}{2\nu_{-1}}$, where $\nu_{-1} = \frac{2\lambda_{-1}(1 -\mu \bar\alpha)}{\mu^2} +\frac{\kappa_{-1}}{\mu}$. 
Set also $w_{-1} =\frac{\lambda_{-1} \nabla f(u_{-1})}{\mu}$. Let $\delta \in \R_{++}$ and $\ell_1, \ell_2, \ell_3 \in (0,1)$. 
Let $N \in \N$ and set $n=0$.

\step{}\label{step:main}
Let $e_n \in \mathcal{H}$, $\theta_n \in \R$, and $\tau_n\in \mathbb{R}_+$ such that $r\tau_n \geq |\theta_n| +\|e_n\| +\delta$. 
Choose $\tilde{L}_n \in [L_{\min}, L_{\max}] $, $\tilde{\alpha}_n \in (0, \bar\alpha]$, and $\tilde{\kappa}_n \in [\kappa_{\min}, \kappa_{\max}]$. Choose $\tilde{\beta}_n \in [0, \bar\beta]$ such that $\tilde{\beta}_n  = 0$ if $c = \bar c$. 
Choose $\tilde{\gamma}_n \in [0, \bar\gamma]$ such that $\tilde{\gamma}_n =0$ if $d=\bar d$ or $c=\bar c$ or $c=0$.
Choose $\tilde{\lambda}_n \in [\lambda_{\min}, \lambda_{\max}]$ such that $\tilde{\lambda}_n  =0$ if $\lambda_{\min} =0$. Set $\beta_n =\tilde{\beta}_n$, $\gamma_n =\tilde{\gamma}_n$, $L_n =\tilde{L}_n$, $\alpha_n =\tilde{\alpha}_n$, $\lambda_n =\tilde{\lambda}_n$, and $\kappa_n = \tilde{\kappa}_n$.
\begin{enumerate}[leftmargin=2em]
\item\label{step:main_comp}
Compute 
\begin{equation*}
\small
\begin{cases}
u_n = x_n + \beta_n(x_n - x_{n-1}),\\ 
v_n = x_n + \gamma_n (x_n -x_{n-1}),\\
w_n \in \prox_{\tau_n \phi}\left(\alpha_n(x_n -v_{n-1}+\lambda_{n-1}\nabla f(u_{n-1}))-\theta_n \frac{\nabla f(u_n)}{\|\nabla f(u_n)\|} -\kappa_n \nabla f(u_n)+ e_n\right ),\\
z_{n+1} = v_n- \lambda_n\nabla f(u_n) +\mu w_n.
\end{cases}
\end{equation*}
\item 
\label{step:main_line}
Set $\rho_n = \frac{1}{\mu^2}(1 -\mu \bar\alpha)\frac{\gamma_n}{\nu_n} + \frac{ L_n\beta_n +\gamma_n}{2 }$ and $\sigma_n =\frac{\mu \alpha_n}{2 \nu_n}$, where $\nu_n = \frac{2 \lambda_n(1 -\mu \bar\alpha)}{\mu^2}+ \frac{\kappa_n }{\mu}$. If
\begin{align}\label{eq:linesearch}
V(z_{n+1}, x_n, w_n,\rho_n, \sigma_n) +c \|z_{n+1}-x_n\|^2 + d \|\nabla f(u_n)\|^2 +\frac{\delta}{\underline{\nu}}\| w_n\|\notag\\
\leq \max_{[n-N]_+ \leq k \leq n}V(x_k, x_{k-1}, w_{k-1},\rho_{k-1}, \sigma_{k-1}),
\end{align}
then go to Step~\ref{step:last}.
\item\label{step:main_mult}
Set $L_n \leftarrow \min\{L_n/\ell_1, L_{\max}\}$, $ \beta_n \leftarrow \ell_{2}\beta_n $,  $ \gamma_n \leftarrow \ell_{2}\gamma_n $, $\alpha_n \leftarrow \ell_{2}\alpha_n $, $\lambda_n \leftarrow \max\{\ell_3 \lambda_n, \lambda_{\min}\}$, $\kappa_n \leftarrow \max\{\ell_3 \kappa_n, \kappa_{\min}\}$, and go to Step~\ref{step:main}\ref{step:main_comp}. 
\end{enumerate}

\step{}\label{step:last}
Set $x_{n+1} \leftarrow z_{n+1}$. If a termination criterion is not met, then let $n\leftarrow n+1$ and go to Step~\ref{step:main}.
\end{algo}
\end{tcolorbox}
\begin{remark}[Simplified AGDNL]
\label{r:lambda0}
In AGDNL algorithm, if $\lambda_{\min} =0$ and $\bar \gamma=0$ then Step~\ref{step:main} reduces to

\newcounter{st}
\renewcommand\thest{\arabic{st}'}
\setcounter{st}{1}
\refstepcounter{st}
\item[~\(\triangleright\)~\textbf{Step~\thest.}]\label{step:main'}
Let $e_n \in \mathcal{H}$, $\theta_n \in \R$, and $\tau_n\in \mathbb{R}_+$ such that $r\tau_n \geq |\theta_n| +\|e_n\| +\delta$. 
Choose $\tilde{L}_n \in [L_{\min}, L_{\max}] $, $\tilde{\alpha}_n \in (0, \bar\alpha]$, and $\tilde{\kappa}_n \in [\kappa_{\min}, \kappa_{\max}]$. Choose $\tilde{\beta}_n \in [0, \bar\beta]$ such that $\tilde{\beta}_n  = 0$ if $c = \bar c$. 
Set $\beta_n =\tilde{\beta}_n$, $\gamma_n =\tilde{\gamma}_n$, $L_n =\tilde{L}_n$, $\alpha_n =\tilde{\alpha}_n$, $\lambda_n =\tilde{\lambda}_n$, and $\kappa_n = \tilde{\kappa}_n$.

\begin{enumerate}
\item\label{step:main_comp'}
Compute 
\begin{equation*}
\begin{cases}
u_n &= x_n + \beta_n(x_n - x_{n-1}),\\ 
w_n &\in \prox_{\tau_n \phi}\left(\alpha_n(x_n -x_{n-1})-\theta_n \frac{\nabla f(u_n)}{\|\nabla f(u_n)\|} -\kappa_n \nabla f(u_n)+ e_n\right ),\\
z_{n+1}\!\!\! &= x_n +\mu w_n.
\end{cases}
\end{equation*}
\item 
Set $\rho_n =  \frac{ L_n\beta_n }{2 }$ and $\sigma_n =\frac{\mu \alpha_n}{2 \nu_n}$, where $\nu_n =  \frac{\kappa_n }{\mu}$. If
\begin{align*}
V(z_{n+1}, x_n, w_n,\rho_n, \sigma_n) +c \|z_{n+1}-x_n\|^2 + d \|\nabla f(u_n)\|^2 + \frac{\delta}{\underline{\nu}}\| w_n\|\\
\leq \max_{[n-N]_+ \leq k \leq n}V(x_k, x_{k-1}, w_{k-1},\rho_{k-1}, \sigma_{k-1}),
\end{align*}
then go to Step~\ref{step:last}.
\item
Set $L_n \leftarrow \min\{L_n/\ell_1, L_{\max}\}$, $ \beta_n \leftarrow \ell_{2}\beta_n $, $\alpha_n \leftarrow \ell_{2}\alpha_n $,  $\kappa_n \leftarrow \max\{\ell_3 \kappa_n, \kappa_{\min}\}$, and go to Step~\ref{step:main'}\ref{step:main_comp'}.
\end{enumerate}
\end{remark}

Before analyzing the convergence of Algorithm~\ref{algo:AGDNL}, we characterize the descent property of our potential function in the following lemma.

\begin{lemma}
\label{l:descent}
Suppose that $f$ is a differentiable function whose gradient is $L$-Lipschitz continuous. Under the setting of AGDNL algorithm, for all $n \in \N$, the following hold:
\begin{enumerate}
\item\label{l:descent_wn}
$\|w_n\|^2 - \mu \alpha_n \langle w_{n-1}, w_n\rangle  +\kappa_n \langle \nabla f(u_n), w_n\rangle  +\delta \| w_n\| \leq 0$. Consequently, $\|w_n\|^2 + (\delta - k_{\max} \|\nabla f(u_n)\| - \mu  \bar \alpha\|w_{n-1}\|) \|w_n\| \leq 0$. 
\item\label{l:descent_Vn}
If $L_n \geq L$, $L_n \beta_n \leq L_{n-1} \beta_{n-1}$, $\gamma_n \leq \gamma_{n-1}$, $\frac{\gamma_n}{\nu_n} \leq \frac{\gamma_{n-1}}{\nu_{n-1}}$, and $\frac{\alpha_n}{\nu_n} \leq \frac{\alpha_{n-1}}{\nu_{n-1}}$, then
\begin{align*}
&V(z_{n+1}, x_n, w_n,\rho_n, \sigma_n) + \frac{c_n}{\nu_n}\|z_{n+1} -x_n||^2 +  \frac{d_n}{\nu_n}\|\nabla f(u_n)\|^2 + \frac{\delta}{\nu_n} \| w_n\| \\
&\leq V(x_n, x_{n-1}, w_{n-1},\rho_{n-1}, \sigma_{n-1}),
\end{align*}  
where
\begin{align}
c_n &= \frac{(1 -\mu \bar\alpha) (1 -2\gamma_n )}{\mu^2} - \frac{(2 L_n\beta_n +\gamma_n + L_n)\nu_n }{2} \label{d:cn}\\
\text{and~~}
d_{n} &= \frac{2(1- \mu \bar\alpha)(\lambda_n^2 -\gamma_n\lambda_n)-\mu\gamma_n\kappa_n+2\mu\lambda_n \kappa_n}{2\mu^2} \label{d:dn}.
\end{align}
\end{enumerate}
\end{lemma}
\begin{proof}
\ref{l:descent_wn}:
Set $y_n = \alpha_n(x_n -v_{n-1}+\lambda_n\nabla f(u_{n-1}))-\theta_n \frac{\nabla f(u_n)}{\|\nabla f(u_n)\|}-\kappa_n \nabla f(u_n)+ e_n$. By the update of $x_n$, we have $\mu w_{n-1}= x_n - v_{n-1} +\lambda_{n-1} \nabla f(u_{n-1})$, and so
\begin{equation*}
y_n = \alpha_n\mu w_{n-1} -\theta_n \frac{\nabla f(u_n)}{\|\nabla f(u_n)\|} -\kappa_n \nabla f(u_n)+ e_n.
\end{equation*}
By Step~\ref{step:main}\ref{step:main_comp}, $w_n \in \prox_{\tau_n \phi}(y_n)$. Since $\phi$ satisfies the dry-like friction property with constant $r$, we have from Lemma~\ref{l:dry} that $\langle w_n-y_n, w_n\rangle  + r\tau_n \|w_n\| \leq 0$, which yields
\begin{equation*}
\|w_n\|^2 - \alpha_n \mu\langle w_{n-1}, w_n\rangle  + \kappa_n \langle \nabla f(u_n), w_n\rangle  + (r\tau_n - |\theta_n| - \|e_n\|)\| w_n\| \leq 0.    
\end{equation*}
Combining this with $r\tau_n \geq |\theta_n| +\|e_n\| +\delta$, we get the first conclusion, which implies the second one.

\ref{l:descent_Vn}: Set $\Delta_n = x_n - x_{n-1}$. It follows from the update of $z_{n+1}$ that
\begin{equation*}
\mu w_n=z_{n+1} -v_n +\lambda_n \nabla f(u_n) 
=(z_{n+1} -x_n)-\gamma_n \Delta_n +\lambda_n \nabla f(u_n),    
\end{equation*}
and so
\begin{align}\label{eq:57}
\mu^2\|w_n\|^2&=  \|z_{n+1} -x_n\|^2 +\gamma_n^2 \|\Delta_n\|^2 +\lambda_n^2 \|\nabla f(u_n)\|^2 - 2\gamma_n \langle z_{n+1} -x_n,  \Delta_n\rangle \notag\\
 & \quad +2  \lambda_n\langle z_{n+1} -x_n, \nabla f(u_n)\rangle -2 \gamma_n\lambda_n\langle  \Delta_n, \nabla f(u_n) \rangle \notag\\
 &\geq  \|z_{n+1} -x_n\|^2 +\gamma_n^2 \|\Delta_n\|^2 +\lambda_n^2 \|\nabla f(u_n)\|^2 - \gamma_n (\|z_{n+1}-x_n\|^2 + \|\Delta_n\|^2)  \notag\\
  & \quad +2 \lambda_n\langle z_{n+1} -x_n, \nabla f(u_n)\rangle - \gamma_n\lambda_n(\|\Delta_n\|^2 + \|\nabla f(u_n)\|^2) \notag\\
  &=  (1-\gamma_n)\|z_{n+1} -x_n\|^2 +(\gamma_n^2-\gamma_n -\gamma_n\lambda_n) \|\Delta_n\|^2 + (\lambda_n^2 -\gamma_n\lambda_n) \|\nabla f(u_n)\|^2 \notag\\
   & \quad +2 \lambda_n\langle z_{n+1} -x_n, \nabla f(u_n)\rangle.
\end{align}
Since $\nabla f$ is $L$-Lipschitz continuous, $L_n \geq L$, and $u_n = x_n + \beta_n(x_n - x_{n-1})$, we have
\begin{align}\label{eq:es2}
\langle \nabla f(u_n), z_{n+1} -x_n \rangle &= \langle \nabla f(u_n) - \nabla f(x_n), z_{n+1} - x_{n} \rangle + \langle \nabla f(x_n), z_{n+1} -x_n \rangle \notag\\
& \geq -L_n\|u_n -x_n\| \|z_{n+1} -x_n\| + f(z_{n+1}) -f(x_n) -\frac{L_n}{2}\|z_{n+1} - x_n\|^2 \notag\\
& \geq  -\frac{L_n\beta_n}{2}\|\Delta_n\|^2 -\frac{L_n\beta_n +L_n}{2}\|z_{n+1} -x_n\|^2 + f(z_{n+1}) -f(x_n).
\end{align}
Combining this with \eqref{eq:57}, we derive that
\begin{align}\label{eq:es1}
 &\|w_n\|^2 - \mu \alpha_n \langle w_{n-1}, w_n\rangle+\kappa_n \langle \nabla f(u_n), w_n\rangle \notag\\
&\geq \|w_n\|^2 - \frac{\mu \alpha_n}{2}(\|w_n\|^2 +\|w_{n-1}\|^2)+\kappa_n \langle \nabla f(u_n), w_n\rangle \notag\\
& \geq  -\frac{\mu \alpha_{n}}{2} \|w_{n-1}\|^2 +\frac{\mu \alpha_n}{2} \|w_n\|^2 +(1-\mu \bar\alpha)\|w_n\|^2\notag\\
&\quad+\frac{\kappa_n }{\mu} \langle \nabla f(u_n), z_{n+1} -x_n \rangle -\frac{\gamma_n\kappa_n }{\mu} \langle \nabla f(u_n), \Delta_n \rangle +\frac{\lambda_n \kappa_n }{\mu}\|\nabla f(u_n)\|^2\notag\\
& \geq  -\frac{\mu \alpha_{n}}{2} \|w_{n-1}\|^2  +\frac{\mu \alpha_n}{2} \|w_n\|^2 \notag\\
& \quad +\frac{1-\mu \bar\alpha }{\mu^2}\left( (1-\gamma_n)\|z_{n+1} -x_n\|^2 +(\gamma_n^2-\gamma_n -\gamma_n\lambda_n) \|\Delta_n\|^2 + (\lambda_n^2 -\gamma_n\lambda_n) \|\nabla f(u_n)\|^2  \right)\notag\\
  & \quad+ \frac{2\lambda_n(1-\mu \bar\alpha)}{\mu^2}\langle z_{n+1} -x_n, \nabla f(u_n)\rangle +\frac{\kappa_n}{\mu} \langle \nabla f(u_n), z_{n+1} -x_n \rangle  \notag\\
  &\quad -\frac{\gamma_n \kappa_n }{2\mu} (\|\nabla f(u_n)\|^2 +\|\Delta_n \|^2) +\frac{\lambda_n\kappa_n }{\mu}\|\nabla f(u_n)\|^2\notag\\
&\geq -\frac{\mu \alpha_{n}}{2} \|w_{n-1}\|^2 +\frac{\mu \alpha_n}{2} \|w_n\|^2 +\bar c_n\|z_{n+1}-x_n\|^2 -\tilde c_n\|\Delta_n\|^2+ d_n\|\nabla f(u_n)\|^2\notag\\
&\quad +\left(\frac{2\lambda_n(1-\mu \bar\alpha)}{\mu^2} +\frac{\kappa_n }{\mu}\right)( f(z_{n+1}) -f(x_n)), 
\end{align}
where the last inequality follows from \eqref{eq:es2}, $d_n$ is defined in \eqref{d:dn} and
\begin{align}
\bar c_n &=\frac{(1-\mu \bar\alpha) (1-\gamma_n)}{\mu^2} -\frac{L_n\beta_n + L_n}{2 }  \left(\frac{2\lambda_n(1-\mu \bar\alpha)}{\mu^2}+ \frac{\kappa_n}{\mu}\right), \notag\\
\tilde c_n  &=\frac{(1 -\mu \bar\alpha)\left(\gamma_n +\gamma_n\lambda_n+L_n\lambda_n \beta_n \right)  + \frac{\mu \kappa_n}{2}(L_n\beta_n +\gamma_n)}{\mu^2} \notag\\
&=\frac{(1 -\mu \bar\alpha)\gamma_n}{\mu^2} + \frac{ L_n\beta_n +\gamma_n}{2 } \left(\frac{2\lambda_n(1-\mu \bar\alpha)}{\mu^2}+ \frac{\kappa_n}{\mu}\right) =\rho_n \nu_n.\label{d:eta} 
\end{align}
It follows from \ref{l:descent_wn}, \eqref{eq:es1}, and the definition of $\nu_n$ that
\begin{align*}
&\frac{\mu \alpha_n}{2} \|w_n\|^2 -\frac{\mu \alpha_n}{2}\|w_{n-1}\|^2  +  \nu_n(f(z_{n+1}) -f(x_n))+ \bar c_n \|z_{n+1} -x_n||^2  - \rho_n \nu_n\|\Delta_{n}\|^2 \notag\\
& + d_n\|\nabla f(u_n)\|^2 + \delta \| w_n\|  \leq 0,
\end{align*}
which implies that
\begin{align}\label{eq:ADFGe_4}
 &f(z_{n+1})+ \rho_n \|z_{n+1} -x_n||^2 + \frac{\mu \alpha_n}{2\nu_n} \|w_n\|^2 + \frac{c_n}{\nu_n}\|z_{n+1} -x_n||^2 +  \frac{d_n}{\nu_n}\|\nabla f(u_n)\|^2 \notag\\
 &\hspace{3cm} + \frac{\delta}{\nu_n} \| w_n\| \leq f(x_n)+ \rho_n \|x_n -x_{n-1}||^2 + \frac{\mu \alpha_n}{2 \nu_n} \|w_{n-1}\|^2.
  \end{align}  
Here, we note that
\begin{align*}
\bar c_n-\rho_n \nu_n&= \frac{(1 -\mu \bar\alpha) (1 -2\gamma_n )}{\mu^2} - \frac{2 L_n\beta_n +\gamma_n + L_n}{2} \left(\frac{2\lambda_n(1-\mu \bar\alpha)}{\mu^2}+ \frac{\kappa_n}{\mu}\right) \\
&= \frac{(1 -\mu \bar\alpha) (1 -2\gamma_n )}{\mu^2} - \frac{(2 L_n\beta_n +\gamma_n + L_n)\nu_n}{2}  =c_n.
\end{align*}
Since $L_n \beta_n \leq L_{n-1}\beta_{n-1}$, $\gamma_n \leq \gamma_{n-1}$, and $\frac{\gamma_n}{\nu_n} \leq \frac{\gamma_{n-1}}{\nu_{n-1}}$, we have $\rho_n \leq \rho_{n-1}$. As $\frac{\alpha_n}{\nu_n} \leq \frac{\alpha_{n-1}}{\nu_{n-1}}$, using \eqref{eq:ADFGe_4} and the definition of $V$ in \eqref{eq:V}, we get the conclusion. 
\end{proof}

It can be seen from Lemma~\ref{l:descent} that the sufficient descent of $V$ is satisfied as long as $L_n$ is sufficiently large and $\gamma_n$, $\beta_n$, and $\alpha_n$ are sufficiently small. Using this result, we show in the following lemma that the line search criterion \eqref{eq:linesearch} in Algorithm~\ref{algo:AGDNL} is well defined.

\begin{lemma}[Well-definedness of the line search criterion]
Suppose that $f$ is a differentiable function whose gradient is $L$-Lipschitz continuous. Then, for all $n \in \N$, the line search criterion \eqref{eq:linesearch} in AGDNL algorithm is satisfied after a finite number of inner iterations.
\end{lemma}
\begin{proof}
Fix $n \in \mathbb{N}$. Since $\ell_1 \in (0, 1)$ and $\tilde{L}_n > L_{\min}$, there exists $k_0$ such that $\ell_1^{-k_0}\tilde L_n\geq L_{\max}$. It also can be seen that there exists $k_1$ such that $\ell_3^{k_1} \tilde{\lambda}_n \leq \lambda_{\min}$ and $\ell_3^{k_1} \tilde{\kappa}_n \leq \kappa_{\min}$ since $\ell_3 \in (0,1)$. Suppose that the line search criterion \eqref{eq:linesearch} is not satisfied after $\max\{k_0, k_1\}$ inner iterations. At the $k$-th inner iteration with $k >\max\{k_0, k_1\}$, by Step~\ref{step:main}\ref{step:main_mult},
\begin{equation}\label{eq:kth}
L_n = L_{\max}, \quad \gamma_n =\ell_{2}^{k} \tilde{\gamma}_n, \quad \alpha_n =\ell_{2}^{k} \tilde{\alpha}_n, \quad \beta_n =\ell_{2}^{k} \tilde{\beta}_n, \quad \lambda_n = \lambda_{\min}, \text{~and~} \kappa_n =\kappa_{\min}.
\end{equation}
Then $\nu_n = \frac{2 \lambda_{\min}(1 -\mu \bar \alpha)}{\mu^2}+ \frac{\kappa_{\min} }{\mu} = \underline{\nu}$, and so $\frac{\gamma_n}{\nu_n} =\frac{\ell_2^k\tilde{\gamma}_n}{\nu_n} =\frac{\ell_2^k\tilde{\gamma}_n}{\underline{\nu}}$ and $\frac{\alpha_n}{\nu_n} =\frac{\ell_2^k\tilde{\alpha}_n}{\nu_n} =\frac{\ell_2^k \tilde{\alpha}_n}{\underline{\nu}}$. Since $0<\ell_2 <1$, there exists $k_2 >\max\{k_0, k_1\}$ such that after $k_2$ inner iterations, $\gamma_n \leq \gamma_{n-1}$, $L_n \beta_n = L_{\max} \ell_{2}^{k} \tilde{\beta}_n  \leq L_{n-1} \beta_{n-1}$,  $\frac{\gamma_n}{\nu_n} \leq \frac{\gamma_{n-1}}{\nu_{n-1}}$, and $\frac{\alpha_n}{\nu_n} \leq \frac{\alpha_{n-1}}{\nu_{n-1}}$.
Using Lemma~\ref{l:descent} and noting that $\nu_n =\underline{\nu}$, we get
\begin{align}
\label{eq:lsV}
&V(z_{n+1}, x_n, w_n,\rho_n, \sigma_n) + \frac{c_n}{\underline{\nu}}\|z_{n+1} -x_n||^2 +  \frac{d_n}{\underline{\nu}}\|\nabla f(u_n)\|^2 + \frac{\delta}{\underline{\nu}} \| w_n\| \notag \\
&\leq V(x_n, x_{n-1}, w_{n-1},\rho_{n-1}, \sigma_{n-1}).
\end{align}
It follows from \eqref{eq:kth} and the definition of $c_n$ in \eqref{d:cn} that
\begin{align}\label{eq:cn}
c_n&= \frac{(1 -\mu \bar\alpha) (1 -2\ell_{2}^{k} \tilde{\gamma}_n)}{\mu^2} - \frac{( 2 \ell_{2}^{k}L_{\max} \tilde{\beta}_n +\ell_{2}^{k}\tilde{\gamma}_n + L_{\max})\underline{\nu}}{2} \notag\\
&= \frac{ 1-\mu \bar\alpha}{\mu^2}- \frac{L_{\max} \underline{\nu}}{2}-\ell_2^k \left[ \frac{2\tilde{\gamma}_n (1-\mu \bar\alpha)}{\mu^2}  + \frac{2L_{\max}\tilde{\beta}_n \underline{\nu} +\tilde{\gamma}_n\underline{\nu}}{2} \right]\notag\\
&=\bar c \underline{\nu} -\ell_2^k \left[ \frac{2\tilde{\gamma}_n (1-\mu \bar\alpha)}{\mu^2}  + \frac{2L_{\max}\tilde{\beta}_n \underline{\nu} +\tilde{\gamma}_n\underline{\nu}}{2} \right].
\end{align}
Since $c \leq \bar c=\frac{2(1-\mu \bar \alpha) - \mu^2 L_{\max} \underline{\nu}}{2 \mu^2 \underline{\nu}}$, we consider two cases: $c <\bar c$ and $c = \bar c$. If $c <\bar c$, using \eqref{eq:cn} and the fact that $0< \ell_{2} <1$, there exists $k_3$ such that after $k_3$ inner iterations, $c_n \geq c\underline{\nu}$. If $c = \bar c$, by the choice of $\gamma_n$, $\beta_n$ in Step~\ref{step:init} and Step~\ref{step:main}, for all $n \in \N$, $\tilde \gamma_n =0$ and $\tilde \beta_n =0$, which together with \eqref{eq:cn} yields $c_n \geq c\underline{\nu}$.

Now, we consider $d_n$. Using definition of $d_n$ and \eqref{eq:kth}, we have 
\begin{align*}
d_{n}&= \frac{2(1- \mu \bar\alpha)(\lambda_n^2 -\ell_{2}^k \tilde{\gamma}_n\lambda_n)-\mu \ell_{2}^{k}\tilde{\gamma}_n\kappa_n+2\lambda_n \mu \kappa_n }{2\mu^2}\notag\\
& = \frac{(1-\mu \bar\alpha)\lambda_{\min}^2 + \lambda_{\min} \mu \kappa_{\min}}{\mu^2} - \ell_2^k\frac{\tilde{\gamma}_n(2(1 -\mu \bar\alpha)\lambda_{\min} +\mu\kappa_{\min})}{2\mu^2} \notag\\
& =\bar d \underline{\nu} - \ell_2^k\frac{\tilde{\gamma}_n(2(1 -\mu \bar\alpha)\lambda_{\min} +\mu\kappa_{\min})}{2\mu^2}.
\end{align*}
By Step~\ref{step:init} and Step~\ref{step:main}, $d \in [0, \bar d]$ and if $d =\bar d$, then $\tilde{\gamma}_n =0$. Since $0<\ell_2 <1$, there exists $k_4$ such that after $k_4$ inner iterations, $d_{n} \geq d \underline{\nu}$. Hence, in view of \eqref{eq:lsV}, after at most $\max\{ k_2,k_3, k_4\}$ inner iterations, the line search criterion \eqref{eq:linesearch} holds.
\end{proof}

We now establish the convergence properties of AGDNL algorithm. In particular, after finite number of iterations, Step~\ref{step:main}\ref{step:main_comp} of AGDNL algorithm reduces to
\begin{equation}\label{AGNL}
\begin{cases}
u_n &= x_n + \beta_n(x_n - x_{n-1}),\\ 
v_n & = x_n + \gamma_n (x_n -x_{n-1}),\\
z_{n+1}\!\!\! &= v_n- \lambda_n\nabla f(u_n)
\end{cases}
\end{equation}
and the line search criterion becomes
\begin{align}\label{AGNL_ls}
V(z_{n+1}, x_n, 0,\rho_n, \sigma_n) +c \|z_{n+1}-x_n\|^2 + d \|\nabla f(u_n)\|^2
\notag\\
\leq \max_{[n-N]_+ \leq k \leq n}V(x_k, x_{k-1}, 0,\rho_{k-1}, \sigma_{k-1}).    
\end{align}
We call this version as the \emph{fast gradient algorithm with nonmonotone line search} (AGNL).

For each $n \in \N$, set $m(n)\in \argmax_{[n-N]_+ \leq k \leq n}V(x_k, x_{k-1}, w_{k-1},\rho_{k-1}, \sigma_{k-1})$. The following technical lemma will be of use in our analysis.

\begin{lemma}
\label{l:limit}
Let $f \colon \mathcal{H} \to \R$ be a continuous differentiable function. Let $(x_n)_{n \in \N}$ and $(u_n)_{n \in \N}$ be sequences in $\mathcal{H}$ such that $(x_n)_{n \in \N}$ is bounded and $x_n -u_n \to 0$ as $n \to +\infty$. Then $f(x_n) -f(u_n) \to 0$ as $n \to +\infty$.
\end{lemma}
\begin{proof}
By the mean value theorem, $f(x_n) -f(u_n) = \langle \nabla f(z_n), x_n -u_n \rangle$, where $z_n = x_n -\zeta_n (x_n -u_n)$ for some $\zeta_n \in [0,1]$. Then $(z_n)_{n \in \N}$ is bounded due to the assumptions on $(x_n)_{n \in \N}$ and $(u_n)_{n \in \N}$. This together with the continuity of $\nabla f$ implies the existence of $M$ such that, for all $n \in \N$, $\|\nabla f(z_n)\| \leq M$, which yields $|f(x_n) -f(u_n) | \leq M\|x_n -u_n\|$. Since $x_n -u_n \to 0$ as $n \to +\infty$, we obtain the conclusion.
\end{proof}

\begin{theorem}[Subsequential convergence and structural simplification]
\label{t:subseq}
Let $(x_n)_{n\in \mathbb{N}}$ be the sequence generated by AGDNL algorithm and $\Omega$ be the set of strong cluster points of $(x_n)_{n \in \N}$. Suppose that $f$ is a differentiable function whose gradient is $L$-Lipschitz continuous, that $f$ is bounded below, and that the set $\{x \in \mathcal{H}: f(x) \leq f(x_0) + \sigma_{-1}\|w_{-1}\|^2\}$ is bounded. Then the following hold:
\begin{enumerate}
\item\label{t:subseq_bounded}
$(x_n)_{n \in \N}$ is bounded.
\item\label{t:subseq_limit} 
As $ n \to +\infty$, $\|x_{n+1}-x_n\| \to 0$,  $\|w_n\| \to 0$, $f(x_n) \to \zeta$, and \\ $V(x_{n}, x_{n-1}, w_{n-1},\rho_{n-1}, \sigma_{n-1}) \to \zeta$ for some $\zeta \in \mathbb{R}$. Moreover, for all $\bar x \in \Omega$, $ f(\bar x) =\zeta$.  
\item\label{t:subseq_simplify} 
Suppose that $\lambda_{\min} >0$. Then $\lim_{n \to +\infty } \nabla f(x_n) =0$ and for all $\bar x \in \Omega$, $\nabla f(\bar x) =0$. Moreover, from some $N_0 \in \N$ onwards, $w_n =0$ and AGDNL algorithm becomes AGNL algorithm.
\end{enumerate}
\end{theorem}
\begin{proof}
\ref{t:subseq_bounded}: First, we prove by induction that, for all $n \in \N$,
\begin{equation}\label{eq:bx1}
V(x_{n}, x_{n-1}, w_{n-1},\rho_{n-1}, \sigma_{n-1}) \leq f(x_0) + \sigma_{-1}\|w_{-1}\|^2.
\end{equation}
Since $x_{-1} =x_0$, we have $V(x_{0}, x_{-1}, w_{-1},\rho_{-1},\sigma_{-1}) = f(x_0) + \sigma_{-1}\|w_{-1}\|^2$, so \eqref{eq:bx1} holds for $n=0$. Suppose that \eqref{eq:bx1} holds for all $n \leq n_0$. 
From \eqref{eq:linesearch}, we obtain that
\begin{align*}
V(x_{n_0+1}, x_{n_0}, w_{n_0},\rho_{n_0}, \sigma_{n_0}) &\leq \max_{[n_0-N]_+ \leq k \leq n_0}V(x_k, x_{k-1}, w_{k-1},\rho_{k-1}, \sigma_{k-1})\notag\\
&\leq   f(x_0) + \sigma_{-1}\|w_{-1}\|^2,
\end{align*}
and hence \eqref{eq:bx1} is satisfied for $n=n_0+1$, which completes the induction. Next, it follows from \eqref{eq:bx1} and the definition of $V$ that $f(x_n) \leq  f(x_0) + \sigma_{-1}\|w_{-1}\|^2$. The conclusion now follows from the boundedness of the set $\{x \in \mathcal{H}: f(x) \leq f(x_0) + \sigma_{-1}\|w_{-1}\|^2\}$.

\ref{t:subseq_limit}: It follows from the line search criterion \eqref{eq:linesearch} that
\begin{align}\label{eq:lsx}
&V(x_{n+1}, x_n, w_n,\rho_n, \sigma_n) +c\|x_{n+1}-x_n\|^2  + d \|\nabla f(u_n)\|^2+ \frac{\delta}{\underline{\nu}}\| w_n\|\notag\\ &\leq \max_{[n-N]_+ \leq k \leq n}V(x_k, x_{k-1}, w_{k-1},\rho_{k-1}, \sigma_{k-1})\notag\\
&= V(x_{m(n)}, x_{m(n)-1}, w_{m(n)-1},\rho_{m(n)-1}, \sigma_{m(n)-1}).
\end{align}
Therefore,
\begin{align*}
&V(x_{m(n+1)}, x_{m(n)}, w_{m(n)},\rho_{m(n)}, \sigma_{m(n)}) \notag\\
&= \max_{[n+1-N]_+ \leq k \leq n+1}V(x_k, x_{k-1}, w_{k-1},\rho_{k-1}, \sigma_{k-1})\notag\\
&=\max \left\{V(x_{n+1}, x_n, w_n,\rho_n, \sigma_n), \max_{[n+1-N]_+ \leq k \leq n}V(x_k, x_{k-1}, w_{k-1},\rho_{k-1}, \sigma_{k-1}) \right\}\notag\\
& \leq  \max_{[n-N]_+ \leq k \leq n}V(x_k, x_{k-1}, w_{k-1},\rho_{k-1}, \sigma_{k-1}) \notag\\
&=V(x_{m(n)}, x_{m(n)-1}, w_{m(n)-1},\rho_{m(n)-1}, \sigma_{m(n)-1}),
\end{align*}
which implies that the sequence $(V(x_{m(n+1)}, x_{m(n)}, w_{m(n)},\rho_{m(n)}, \sigma_{m(n)}))_{n\in \N}$ is nonincreasing. On the other hand, this sequence is bounded below due to the bounded below property of $f$ and the definition of $V$. Thus, there exists $\zeta$ such that
\begin{equation}\label{eq:limV}
\zeta =\lim_{n \to +\infty}V(x_{m(n)}, x_{m(n)-1}, w_{m(n)-1},\rho_{m(n)-1}, \sigma_{m(n)-1}).
\end{equation} 
In \eqref{eq:lsx}, replacing $n$ by $m(n)-1$ yields
\begin{align*}
&V(x_{m(n)}, x_{m(n)-1}, w_{m(n)-1},\rho_{m(n)-1}, \sigma_{m(n)-1}) +c\|x_{m(n)}-x_{m(n)-1}\|^2 \notag\\ 
&+d \|\nabla f(u_{m(n)-1})\|^2 +\frac{\delta}{\underline{\nu}}\| w_{m(n)-1}\| \notag\\
&\leq V(x_{m(m(n)-1)}, x_{m(m(n)-1)-1}, w_{m(m(n)-1)-1},\rho_{m(m(n)-1)-1}, \sigma_{m(m(n)-1)-1}).
\end{align*}
Letting $n \to +\infty$ and noting that $m(n) \geq [n-N]_{+}$, we get $ \lim_{n \to +\infty}\|w_{m(n)-1} \| = 0$. 
If $c >0 $, we also have $\lim_{n \to +\infty}\|x_{m(n)}-x_{m(n)-1}\|~=~0$. If $c=0$ and $d >0$, then $\lim_{n \to +\infty}\|\nabla f(u_{m(n)-1}) \| = 0$, which implies that 
\begin{equation*}
 |x_{m(n)} - x_{m(n)-1}\| =\|-\lambda_{m(n)-1}\nabla f(u_{m(n)-1}) + \mu w_{m(n)-1} \| \to 0 \text{~as~} n \to +\infty   
\end{equation*}
If $c=d=0$, then $\lambda_{\min} =0$ since  $d >0$ if $\lambda_{\min} >0$ and $c=0$. By Step~\ref{step:main}, for all $n \in \N$, $\lambda_n = \gamma_n =0$, and so $\lim_{n \to +\infty}\|x_{m(n)}-x_{m(n)-1}\| =\lim_{n \to +\infty}\mu \|w_{m(n)-1} \| =0$. Therefore, in any cases,
\begin{equation}\label{eq:indu1}
\lim_{n \to +\infty}\|x_{m(n)}-x_{m(n)-1}\| = 0 \text{~and~} \lim_{n \to +\infty}\|w_{m(n)-1} \| = 0.
\end{equation}
By the definitions of $\rho_n$ and $\sigma_n$, for all $n \in \N$,
$0\leq \rho_n \leq  \frac{1}{\mu^2}(1 -\mu \bar\alpha)\frac{\bar\gamma}{\underline{\nu}} + \frac{ L_{\max}\bar\beta +\bar\gamma}{2 }$ and $ 0\leq  \sigma_n \leq\frac{\mu \alpha_{\max}}{2 \underline{\nu}}.$ Combining with \eqref{eq:limV} and \eqref{eq:indu1}, we derive that
\begin{align*}
f(x_{m(n)})&=V(x_{m(n)}, x_{m(n)-1}, w_{m(n)-1},\rho_{m(n)-1}, \sigma_{m(n)-1}) \notag\\
&\quad - \rho_{m(n)-1}\|x_{m(n)}-x_{m(n)-1}\|^2 - \sigma_{m(n)-1} \| w_{m(n)-1}\|^2
\end{align*}
also tends to $\zeta$ as $n \to +\infty$. Since $(x_n)_{n\in \N}$ is bounded and $\lim_{n \to +\infty}\|x_{m(n)}-x_{m(n)-1}\| = 0$, Lemma~\ref{l:limit} implies that 
\begin{equation}\label{eq:indu2}
\lim_{n \to +\infty}f(x_{m(n)-1})=\lim_{n \to +\infty}f(x_{m(n)}) = \zeta.
\end{equation}
Now, we prove by induction that for all $k \in [1, m(n)]$,
\begin{equation}\label{eq:indu}
\lim_{n \to +\infty}\|x_{m(n)-k+1}-x_{m(n)-k}\|=0 , \lim_{n \to +\infty} \|w_{m(n)-k}\|=0, \text{~and~} \lim_{n \to +\infty} f(x_{m(n)-k}) = \zeta.
\end{equation}
It follows from \eqref{eq:indu1} and \eqref{eq:indu2} that \eqref{eq:indu} holds for $k=1$. Suppose that \eqref{eq:indu} holds for $k=k_0$. Replacing $n$ by $ m(n)-k_0-1:= n_m$ in \eqref{eq:lsx} yields
\begin{align*}
&V(x_{m_n+1}, x_{m_n}, w_{m_n},\rho_{m_n}, \sigma_{m_n}) +c\|x_{m_n+1}-x_{m_n}\|^2 +d \|\nabla f(u_{m_n})\|^2+\delta\| w_{m_n}\| \notag\\
&\leq V(x_{m(m_n)}, x_{m(m_n)-1}, w_{m(m_n)-1},\rho_{m(m_n)-1},\sigma_{m(m_n)-1}).
\end{align*}
Using the definition of $V$, we have
\begin{align*}
&(c+ \rho_{m_n})\|x_{m_n+1}-x_{m_n}\|^2  +(\delta + \sigma_{m_n})\| w_{m_n}\| +d \|\nabla f(u_{m_n})\|^2 \notag\\
&\leq V(x_{m(m_n)}, x_{m(m_n)-1}, w_{m(m_n)-1},\rho_{m(m_n)-1},\sigma_{m(m_n)-1}) - f(x_{m_n+1}).
\end{align*}
Similar to the proof of \eqref{eq:indu1}, by the induction hypothesis (\eqref{eq:indu} holds with $k=k_0$) and \eqref{eq:limV}, we obtain that \eqref{eq:indu} holds for $k=k_0+1$, which completes the induction.

Next, let $n \geq N+1$. By the definition of $m(n)$, we have $n -N \leq m(n) \leq n$, and so $k_n := m(n) -(n -N) +1 \in [1, m(n)]$. Then $n-N = m(n) -k_n +1$ and by \eqref{eq:indu}, 
\begin{align*}
&\lim_{n \to +\infty}\|x_n-x_{n-1}\| = \lim_{n \to +\infty}\|x_{m(n)-k_n +1}-x_{m(n)-k_n}\| =0, \\
&\lim_{n \to +\infty}\|w_{n}\|=\lim_{n \to +\infty}\|w_{n-N-1}\| =\lim_{n \to +\infty}\|w_{m(n)-k_n}\|= 0, \text{~and}\\
&\lim_{n \to +\infty} f(x_{n}) =\lim_{n \to +\infty} f(x_{n-N-1}) = \lim_{n \to +\infty} f(x_{m(n)-k_n }) = \zeta.
\end{align*}
Combining this with the definition of $V$ and the boundedness of $\rho_n$ and $\sigma_n$, we deduce that
\begin{equation}\label{eq:limitf}
\lim_{n \to +\infty}V(x_{n}, x_{n-1}, w_{n-1},\rho_{n-1}, \sigma_{n-1}) = \lim_{n \to +\infty}f(x_n) = \zeta.
\end{equation}
Take an arbitrary $\bar x \in \Omega$. Then there exists a sequence $(x_{k_n})_{n \in \N}$ such that $x_{k_n} \to \bar x$ as $n \to +\infty$. Since $f$ is continuous, we have $\lim_{n \to +\infty}f(x_{k_n}) =f(\bar x)$. This, together with \eqref{eq:limitf}, yields that $f(x_n) \to f(\bar x)$ and $V(x_{n}, x_{n-1}, w_{n-1},\rho_{n-1}, \sigma_{n-1}) \to f(\bar x)$ as $n \to +\infty$.

\ref{t:subseq_simplify}: Since $\lambda_{n} \nabla f(u_n) = v_n- x_{n+1} + \mu w_n = -(x_{n+1} -x_n) + \gamma_n (x_n -x_{n-1}) + \mu w_n$, we have that
\begin{align*}
\lambda_{\min} \|\nabla f(u_n)\|\leq \lambda_n \|\nabla f(u_n)\| 
\leq \|x_{n+1} -x_n\| +\bar\gamma\|x_n -x_{n-1}\| + \mu \|w_n\|.
\end{align*}
Combining with \ref{t:subseq_limit} yields $\lim_{n \to +\infty} \|\nabla f(u_n)\|=0$. Since $\nabla f$ is $L$-Lipschitz continuous, it holds that
\begin{equation*}
0\leq \left|\|\nabla f(u_n) \|-\|\nabla f(x_n)\|\right| \leq \|\nabla f(u_n) -\nabla f(x_n)\| \leq L\|u_n -x_n\| = L\beta_n\|x_n -x_{n-1}\|
\end{equation*}
which implies that $\lim_{n \to +\infty} \|\nabla f(x_n)\| =0$, and so $\lim_{n \to +\infty} \|\nabla f(\bar x)\| =0$ any $\bar x \in \Omega$. 

It follows from $\lim_{n \to +\infty}\|\nabla f(u_n)\| =0$ and $\lim_{n \to +\infty}\|w_n\| =0$ that there exists $N_0 \in \mathbb{N}$ such that, for all $n \geq N_0$, $\kappa_{\max} \|\nabla f(u_n)\| +\mu \bar\alpha \|w_{n-1}\| \leq \delta$. Combining this with Lemma~\ref{l:descent}\ref{l:descent_wn}, we have for all $n \geq N_0$, $w_n =0$, which yields $ x_{n+1} -v_n +\lambda_n \nabla f(u_n)=0$, where $u_n = x_n +\beta_n(x_n-x_{n-1})$ and $v_n = x_n +\gamma_n(x_n -x_{n-1})$.
\end{proof}

\begin{remark}[Estimation when $w_n$ vanishes]
Regarding AGDNL algorithm with $N =0$ and $d>0$, we now give the estimation for $n_0$ such that $w_{n_0} =0$. Let $f^* =\min_{x \in \mathcal{H}} f(x)$. It follows from \eqref{eq:linesearch} that
\begin{align*}
V(x_{n+1}, x_n, w_n,\rho_n, \sigma_n)+c \|x_{n+1}-x_n\|^2 +d\|\nabla f(u_n)\|^2+ \frac{\delta}{\underline{\nu}}\|w_n\|
  \notag\\
   \leq V(x_n, x_{n-1}, w_{n-1},\rho_{n-1}, \sigma_{n-1}).
\end{align*}
Using telescoping, we have
\begin{equation}\label{eq:contr_1}
c\sum_{n=0}^{+\infty} \|x_{n+1}-x_n\|^2 +d\sum_{n=0}^{+\infty}\|\nabla f(u_n)\|^2  +\frac{\delta}{\underline{\nu}}\sum_{n=0}^{+\infty}\| w_n\| 
\leq V(x_0, x_{-1}, w_{-1},\rho_{-1}, \sigma_{-1}) - f^*,
\end{equation}
since $V(x_{n+1}, x_n, w_n,\rho_n, \sigma_n) \geq f(x_{n+1}) \geq  f^*$. Set $V_0 =V(x_0, x_{-1}, w_{-1},\rho_{-1}, \sigma_{-1}) - f^*$. Choose $\varepsilon_1 = \frac{ \delta }{\kappa_{\max}  +\mu \bar\alpha }$ and $\bar N \geq \max\{\frac{V_0}{d \varepsilon_1^2}, \frac{\underline{\nu} V_0}{\delta \varepsilon_1}\} $. 
We see that there exists $ n_0 \leq \bar N$ such that $\|w_{n_0-1}\| \leq \varepsilon_1$ and $\|\nabla f(u_{n_0})\| \leq \varepsilon_1$. Indeed, suppose that for all $ n \leq \bar N$, $\|w_{n}\| > \varepsilon_1$ or $\|\nabla f(u_n)\| > \varepsilon_1$. Then
\begin{equation*}
d \sum_{n=0}^{n=\bar N} \|\nabla f(u_n)\|^2+\frac{\delta}{\underline{\nu}} \sum_{n=0}^{n=\bar N}\|w_n\| \geq \bar N \min\{d \varepsilon_1^2, \frac{\delta \varepsilon_1}{\underline{\nu}}  \} \geq V_0,
\end{equation*}
which contradicts \eqref{eq:contr_1}. Therefore, $\kappa_{\max} \|\nabla f(u_{n_0})\| +\mu \bar\alpha  \|w_{n_0-1}\| \leq  \kappa_{\max} \varepsilon_1 +\mu \bar\alpha  \varepsilon_1 =\delta$. This, together with Lemma~\ref{l:descent}\ref{l:descent_wn}, yields that $w_{n_0} =0$.
\end{remark}

In the case $\lambda_{\min} =0$, $N=0$ and $\bar \gamma =0$, we show that the iterative sequence is strongly convergent to an ``approximate'' critical point $x_{\infty}$ of $f$ in the sense that $-\nabla f(x_{\infty}) \in \partial \phi(0)$. 
\begin{proposition}[Convergence to an ``approximate'' critical point]
\label{p:l0N0}
Let $(x_n)_{n\in \mathbb{N}}$ be the sequence generated by AGDNL algorithm with $\lambda_{\min} =0$, $N=0$ and $\bar \gamma =0$. Suppose that $f$ is a differentiable function whose gradient is $L$-Lipschitz continuous and that $f$ is bounded below. Then the sequence $(x_n)_{n\in \mathbb{N}}$ strongly converges to a point $x_{\infty} \in \mathcal{H}$. Moreover, 
\begin{enumerate}
\item \label{p:l0N0i}
If $\kappa_n = \kappa_{\min}$, $\tau_n=\tau$ and as $n \to +\infty$, $\theta_n \to 0$ and $e_n \to 0$, then $-\frac{\kappa_{\min}}{\tau}\nabla f(x_{\infty}) \in \partial \phi(0)$. 
\item \label{p:l0N0ii}
If $\|\nabla f(x_{\infty})\| <\frac{\delta}{\kappa_{\max}}$, then there exists $N_0 \in \mathbb{N}$ such that for all $n \geq N_0$, $x_{n} =x_{N_0}$.
\end{enumerate}
\end{proposition}
\begin{proof}
 Since $\lambda_{\min} =0$, $N=0$, and $\bar \gamma =0$, it follows from Remark \ref{r:lambda0} that
\begin{equation*}
V(x_{n+1}, x_n, w_n,\rho_n, \sigma_n) +c \|x_{n+1}-x_n\|^2 +\frac{\delta}{\underline{\nu}}\| w_n\| \leq V(x_n, x_{n-1}, w_{n-1},\rho_{n-1}, \sigma_{n-1}),
\end{equation*}
and $\mu w_n = x_{n+1} -x_n$. Using telescoping, we obtain that $\sum_{n=0}^{\infty}\|x_{n+1} -x_n\|=\mu \sum_{n=0}^{\infty} \|w_n\| < +\infty$, which implies the existence of $x_{\infty}$ such that $x_n \to  x_{\infty}$ as $n \to +\infty$. 

Suppose that $\kappa_n = \kappa_{\min}$, $\tau_n=\tau$, $r\tau \geq |\theta_n| +\|e_n\| +\delta$ and as $n \to +\infty$, $\theta_n \to 0$, and $e_n \to 0$. Since $w_n \in \prox_{\tau_n \phi}\left(\alpha_n(x_n -x_{n-1})-\theta_n \frac{\nabla f(u_n)}{\|\nabla f(u_n)\|} -\kappa_n \nabla f(u_n)+ e_n\right)$, we have
\begin{equation*}
\frac{1}{\tau_n}(\alpha_n(x_n -x_{n-1})-\theta_n \frac{\nabla f(u_n)}{\|\nabla f(u_n)\|} -\kappa_n \nabla f(u_n)+ e_n-w_n) \in \partial \phi(w_n).
\end{equation*}
Letting $n \to +\infty$, we obtain that $-\frac{\kappa_{\min}}{\tau}\nabla f(x_{\infty}) \in \partial \phi(0)$. 

Now suppose that $\kappa_{\max}\|\nabla f(x_{\infty})\| < \delta$. We have $u_n = x_n + \beta_n(x_n -x_{n-1}) \to x_{\infty}$ and $w_n \to 0$ as $n \to +\infty$, so $\kappa_{\max}\|\nabla f(u_n)\| + \mu \bar\alpha\|w_{n-1}\| \xrightarrow{n \to +\infty} \kappa_{\max}\|\nabla f(x_{\infty})\|< \delta$.
Therefore, there exists $N_0 \in \mathbb{N}$ such that, for all $n \geq N_0$, $\kappa_{\max}\|\nabla f(u_n)\| + \mu \bar\alpha\|w_{n-1}\| \leq \delta.$ Combining this with Lemma~\ref{l:descent}\ref{l:descent_wn}, for all $n \geq N_0$, $w_n =0$, and so $x_{n+1} =x_n$.
\end{proof}

\begin{remark}[Finite convergence]
Let $\phi: \mathcal{H} \to \R$ be given by $\phi =r \|x\|$ with $r>0$. Then $\partial \phi(0) = B(0,r)$. Let $(x_n)_{n\in \mathbb{N}}$ be the sequence generated by AGDNL algorithm with $\lambda_{\min} =0$, $N=0$, $\kappa_n = \kappa_{\min} = \kappa_{\max} = \kappa$, $\tau_n= \tau$, $\theta_n \equiv 0$, $e_n \equiv 0$, and $\delta= r\tau$. Then according to Proposition~\ref{p:l0N0}\ref{p:l0N0i}, we have $\|\nabla f(x_{\infty})\| \leq \frac{r \tau}{\kappa}$. Therefore, the smaller $\tau$ is, the closer $x_{\infty}$ is to the critical point of $f$. Moreover, if $\|\nabla f(x_{\infty})\| < \frac{r \tau}{\kappa}$, then $(x_n)_{n \in \N}$ is finitely convergent.
\end{remark}

We now recall the Kurdyka--{\L}ojasiewicz (KL) property \cite{Kur98, Loj63} which plays an important role in our convergence analysis. Let $h: \mathcal{H} \to (-\infty, +\infty]$ be a proper lower semicontinuous function. We say that $h$ satisfies the \emph{KL property} at $\bar x \in \dom \partial h$ if there exist $\varepsilon \in \mathbb{R}_{++}$, $\nu \in (0, +\infty]$, and a continuous concave function $\varphi: [0, \nu) \to \R_+$ such that $\varphi(0) =0$ and $\varphi$ is continuously differentiable on  $(0, \nu)$ with $\varphi' >0$, and for all $x \in B(\bar x; \varepsilon)$ with $h(\bar x) < h(x) < h (\bar x) + \nu$, one has
\begin{equation*}
\varphi'(h(x)- h(\bar x))\dist(0, \partial h(x)) \geq 1.
\end{equation*}
If $h$ satisfies the KL property at each point in $\dom \partial h$, then $h$ is called a \emph{KL function}.

The function $h$ is said to satisfy the \emph{KL property at $\bar x$ with exponent $\theta$} if it satisfies the KL property at $\bar x \in \dom \partial h$ with corresponding function $\varphi(s) = \gamma s^{1-\theta}$ for some $\gamma \in \R_{++}$ and $\theta \in [0,1)$, i.e., there exist $c \in \mathbb{R}_{++}$, $\varepsilon \in \mathbb{R}_{++}$, and $\nu \in (0, +\infty]$ such that
\begin{equation*}
\dist(0, \partial h(x)) \geq c (h(x) -h(\bar x))^{\theta}
\end{equation*}
whenever $x \in B(\bar x; \varepsilon)$ with $h(\bar x) < h(x) < h(\bar x) + \nu$. If $h$ is a KL function and has the same exponent $\theta$ at any $\bar x \in \dom \partial h$, then $h$ is called a \emph{KL function with exponent $\theta$}. Some typical KL functions include strongly convex functions, real analytic functions, and semialgebraic functions, see \cite{ABRS10, Loj63}.

\begin{lemma}[Abstract convergence under KL property]
\label{l:KL}
Let $\mathcal{H}$ and $\mathcal{K}$ be two finite-dimensional real Hilbert
spaces. Let $h\colon \mathcal{K}\to \R$ be a continuous function, $(x_n)_{n\in \mathbb{N}}$ and $(z_n)_{n\in \mathbb{N}}$ be sequences in $\mathcal{H}$ and $\mathcal{K}$, respectively, $(\Delta_n)_{n\in \mathbb{N}}$ be a sequence of nonnegative real numbers, and $I$ be a finite set of integer indices with the smallest index not exceeding $1$. Set $\Delta_k =0$ for $k <0$ and consider the following conditions:
\begin{enumerate}
\renewcommand\theenumi{(\alph{enumi})}
\renewcommand{\labelenumi}{\rm (\alph{enumi})}
\item\label{a:decrease}
There exists a positive real number $\alpha$ such that, for all $n \in \N$, $h(z_{n+1}) + \alpha \Delta_n^2 \leq  h(z_n)$.
\item\label{a:error}
There exist nonnegative real numbers $\beta_i$, $i\in I$ such that, for all $n \in \N$, $\dist(0, \partial h(z_n)) \leq \sum_{i\in I} \beta_i\Delta_{n-i}$.
\item\label{a:distance}
There exists an integer $j$ such that, for all $n \in \N$, $\|x_{n+1} -x_n\| \leq  \Delta_{n+j}$.
\end{enumerate}
Suppose that \ref{a:decrease} and \ref{a:error} hold, that the sequence $(z_n)_{n \in \N}$ is bounded with $\Omega$ being the set of cluster points, and that $h$ satisfies the KL property at any point of $\Omega$. Then the following hold:
\begin{enumerate}
\item
For all $\bar z \in \Omega$, $0 \in \partial h(\bar z)$ and $h(z_n)\downarrow \bar h := h(\overline{z})$ as $n\to +\infty$.
\item 
If \ref{a:distance} holds, then $\sum_{n=0}^{+\infty} \|x_{n+1} -x_n\| < +\infty$ and the sequence $(x_n)_{n \in \N}$ converges to a point $\bar x \in \mathcal{H}$.
\item 
Suppose further that $h$ satisfies the KL property at any point of $\Omega$ with exponent $\alpha \leq 1/2$. Then there exist $\Gamma_1 \in \R_{++}$ and $\zeta \in (0,1)$ such that, for all $n \in \N$, $0 \leq h(z_n) - \bar h \leq \Gamma_1 \zeta^{n}$. Moreover, if \ref{a:distance} holds, then there exists $\Gamma_2 \in \R_{++}$ such that, for all $n \in \N$, $\|x_n -\bar x\| \leq \Gamma_2 \zeta^{n/2}.$
\end{enumerate}
\end{lemma}
\begin{proof}
In view of \cite[Lemma~5.1 and Theorem~5.2]{BDL21}, we only need to check that (I) there  exist a subsequence $(z_{k_n})_{n \in \N}$ and $\tilde{z}$ such that $z_{k_n} \to \tilde{z}$ and $h(z_{k_n}) \to h(\tilde{z})$ as $n \to +\infty$; and (II) $h$ is constant on $\Omega$. The first one is due to the continuity of $h$ and the boundedness of $(z_{n})_{n \in \N}$. We now check (II). It follows from \ref{a:decrease} that $(h(z_n))_{n \in \N}$ is nonincreasing, and hence convergent due to (I). Take an arbitrary $\bar z$ in $\Omega$. Then, there exists a subsequence $(z_{k_n})_{n \in \N}$ such that $z_{k_n} \to z$ as $n \to +\infty$. Since $h$ is continuous, we obtain that $h(\bar z) =\lim_{n \to +\infty} h(z_{k_n}) =\lim_{n\to +\infty} h(z_n)$, and so $h$ is constant on $\Omega$. The proof is complete.
\end{proof}

\begin{theorem}[Sequential and linear convergence]
\label{t:fullseq}
Suppose that $\mathcal{H}$ is finite dimensional, that $f$ is a differentiable function whose gradient is $L$-Lipschitz continuous, and that $f$ is bounded below. Let $(x_n)_{n\in \mathbb{N}}$ be the sequence generated by AGDNL algorithm with $\lambda_{\min} >0$. Suppose further that the set $\{x \in \mathcal{H}: f(x) \leq f(x_0) + \sigma_{-1}\|w_{-1}\|^2\}$ is bounded, and that one of the following holds
\begin{enumerate}
\item \label{ADFGe:i}  $N=0$ and $ g(x,y,\rho):= f(x)+ \rho  \|x -y\|^2 $ satisfies the KL property at $(\bar x, \bar x, \rho)$ for all $\bar x \in \mathcal{H}$ and $\rho \in [0, \bar\rho]$, where $\bar\rho: =\frac{1}{\mu^2}(1 -\mu \bar\alpha)\frac{\bar\gamma}{\underline{\nu}} + \frac{ L_{\max}\bar\beta +\bar\gamma}{2 }$.
\item \label{ADFGe:ii} $\frac{1}{\lambda_{\max} }- L \bar\beta -\frac{\bar\gamma}{\lambda_{\min}}-\frac{L}{2} >0$ and $h(x,y) =f(x) +(\frac{L \bar\beta}{2} +\frac{\bar\gamma}{2 \lambda_{\min} }) \|x-y\|^2$ satisfies the KL property at $(\bar x, \bar x)$  for all $\bar x \in \mathcal{H}$.
\end{enumerate}
Then $\sum_{k=0}^{+\infty} \|x_{n+1} -x_n\| < +\infty$ and $(x_n)_{n\in \N}$ converges to a critical point $x^*$ of $f$. Moreover, if the KL property in \ref{ADFGe:i} or \ref{ADFGe:ii} is satisfied with exponent $\theta \leq 1/2$, then there exist $\Gamma \in \R_{++}$ and $\zeta \in (0,1)$ such that, for all $n \in \N$, $\|x_n -x^*\| \leq \Gamma{\zeta}^{n/2}$ and $|f(x_n) -f(x^*) |\leq \Gamma{\zeta}^n$.
\end{theorem}
\begin{proof}
\ref{ADFGe:i}: According to Theorem~\ref{t:subseq}\ref{t:subseq_simplify}, there exists $N_0$ such that for all $n \geq N_0$, $w_{n} =0$ and $x_{n+1}=v_{n}-\lambda_n \nabla f(u_n)$. Let $n \geq N_0+1$. 
We first see that there exists $\underline{c} >0$ such that
\begin{equation}\label{eq:KL1}
V(x_{n+1}, x_n,0,\rho_n, \sigma_n)+ \underline{c} \|x_{n+1}-x_n\|^2  \leq V(x_n, x_{n-1},0,\rho_{n-1},\sigma_{n-1}).
\end{equation}
Indeed, since $N=0$, it follows from \eqref{eq:linesearch} that
\begin{equation}\label{eq:KL1'}
V(x_{n+1}, x_n,0,\rho_n, \sigma_n)+ c\|x_{n+1}-x_n\|^2  + d \|\nabla f(u_n)\|^2 \leq V(x_n, x_{n-1},0,\rho_{n-1},\sigma_{n-1}).
\end{equation}
By Step~\ref{step:init}, $d >0$ if $\lambda_{\min} >0$ and $c=0$. Thus, we must have $c >0$ or $d>0$ since $\lambda_{\min} >0$. If $c>0$, then \eqref{eq:KL1} holds with $\underline{c}:= c$. If $c=0$ and $d >0$, then by Step~\ref{step:main}, for all $n \in \N$, $\gamma_n =0$. Thus $x_{n+1} =x_n -\lambda_n \nabla f(u_n)$, which implies that $\|\nabla f(u_n)\| = \frac{1}{\lambda_n}\|x_{n+1} -x_n\| \geq \frac{1}{\lambda_{\max}}\|x_{n+1} -x_n\| $. This together with \eqref{eq:KL1'} implies that \eqref{eq:KL1} is satisfied with $\underline{c} := d/\lambda_{\max}^2$.

Set $\mathbf z_n = (x_{n}, x_{n-1}, \rho_{n-1})$. Noting that $g(x,y,s)=V(x,y,0,s,t)$, we have from \eqref{eq:linesearch} that
\begin{equation}\label{eq:KL_1}
g(\mathbf z_{n+1})+ \underline{c} \|x_{n+1}-x_n\|^2  \leq g(\mathbf z_n).
\end{equation}
On the other hand,
\begin{equation*}
\nabla g(\mathbf z_{n}) =(2\rho_{n-1}(x_{n} -x_{n-1}) + \nabla f(x_{n}), -2\rho_{n-1}(x_{n} -x_{n-1}) ,\|x_{n} -x_{n-1}\|^2 ).
\end{equation*}
Using the definition of $\rho_n$ and the fact that $\gamma_n \leq \bar\gamma$, $\nu_n \geq \underline{\nu}$, $L_n \leq L_{\max}$, we have
\begin{equation*}
\rho_n = \frac{1}{\mu^2}(1 -\mu \bar\alpha)\frac{\gamma_n}{\nu_n} + \frac{ L_n\beta_n +\gamma_n}{2 }\leq \frac{1}{\mu^2}(1 -\mu \bar\alpha)\frac{\bar\gamma}{\underline{\nu}} + \frac{ L_{\max}\bar\beta +\bar\gamma}{2 }=\bar\rho.
\end{equation*}
Therefore, by using the sum extension of the natural norm to the product space,
\begin{align}\label{eq:d1}
\|\nabla g(\mathbf z_{n})) \|&\leq \|2\rho_{n-1}(x_{n} -x_{n-1}) +\nabla f(x_{n})\| +\|2\rho_{n-1}(x_{n} -x_{n-1})\| +\|x_{n} -x_{n-1}\|^2 \notag\\
&\leq 4 \bar\rho\|x_{n} -x_{n-1}\| +\|\nabla f(x_{n})\| +\|x_{n} -x_{n-1}\|^2.
\end{align}
Since $\nabla f$ is $L$-Lipschitz continuous and $\beta_n \leq \bar\beta$,
\begin{equation*}
\|\nabla f(u_n) -\nabla f(x_n)\| \leq L \|u_n -x_n\| = L\|\beta_n(x_n -x_{n-1})\| \leq L \bar\beta\|x_n -x_{n-1}\|.
\end{equation*}
As $\lambda_n \nabla f(u_n) = v_n - x_{n+1} = -(x_{n+1} -x_n) + \gamma_n (x_n - x_{n-1})$, it follows that
\begin{align}\label{eq:d2}
\|\nabla f(x_n)\|& \leq \|\nabla f(u_n)\| + \|\nabla f(u_n) -\nabla f(x_n)\|\notag\\
&\leq \frac{1}{\lambda_{\min}} (\|x_{n+1}-x_n\| + \bar\gamma\|x_n-x_{n-1}\|)  + L \bar\beta\|x_n -x_{n-1}\|  \notag\\
& =  \frac{1}{\lambda_{\min}}\|x_{n+1}-x_n\|  + \left( \frac{\bar\gamma}{\lambda_{\min}} +  L \bar\beta\right)\|x_n - x_{n-1}\|.
\end{align}
According to Theorem~\ref{t:subseq}\ref{t:subseq_limit}, $\lim_{n \to +\infty}\|x_{n}-x_{n-1}\| =0$, so there exists $N_1 \geq N_0 +1$ such that, for all $n \geq N_1$, $\|x_{n}-x_{n-1}\| \leq 1$, which implies that  $\|x_{n}-x_{n-1}\|^2 \leq \|x_{n}-x_{n-1}\|$. Combining this with \eqref{eq:d1} and \eqref{eq:d2}, there exists $C$ such that, for all $n \geq N_1$,
\begin{equation} \label{eq:KL2}
\| \nabla g(\mathbf z_{n}) \| \leq C(\|x_{n+1}-x_{n}\|+\|x_{n} -x_{n-1}\|).
\end{equation}
Moreover, $(\mathbf z_n)_{n \in \N}$ is bounded due to the boundedness of $(x_n)_{n \in \N}$ and $(\rho_n)_{n \in \N}$. Combining this with \eqref{eq:KL_1} and \eqref{eq:KL2}, we get the conclusion due to Lemma~\ref{l:KL}. 

\ref{ADFGe:ii}: As $\nabla f$ is $L$-Lipschitz continuous, for all $ n \in \N$,
\begin{align*}
&f(x_{n+1}) -f(x_n) -\frac{L}{2}\|x_{n+1}-x_n\|^2 \notag\\
&\leq \langle \nabla f(x_n), x_{n+1} -x_n\rangle  \notag\\
& = \langle\nabla f(x_n) -\nabla f(u_n), x_{n+1} -x_n\rangle + \langle \nabla f(u_n), x_{n+1} -x_n\rangle  \notag\\
& \leq L\|x_n -u_n\|\|x_{n+1} -x_n\| + \frac{1}{\lambda_n}\langle v_n -x_{n+1}, x_{n+1} -x_n \rangle \notag\\
& = L \beta_n \|x_n -x_{n-1}\|\|x_{n+1} -x_n\|- \frac{1}{\lambda_n}\| x_{n+1} -x_{n}\|^2 + \frac{\gamma_n}{\lambda_n}\langle x_n -x_{n-1}, x_{n+1} -x_n \rangle.
\end{align*}
Noting that $\langle x_n -x_{n-1}, x_{n+1} -x_n \rangle \leq \frac{1}{2}\left( \|x_n -x_{n-1}\|^2 + \|x_{n+1} -x_n\|^2 \right)$, we derive that
\begin{equation*}
f(x_{n+1}) -f(x_n) +\left(\frac{1}{\lambda_n} - \frac{L \beta_n}{2} -\frac{\gamma_n}{2\lambda_n}-\frac{L}{2} \right)\|x_{n+1}-x_n\|^2 \leq (\frac{L \beta_n}{2} +\frac{\gamma_n}{2\lambda_n})\|x_n - x_{n-1}\|^2.
\end{equation*}
Using the definition of $h$, we have
\begin{equation}\label{eq:cvKL2}
h(x_{n+1}, x_n) + \hat c \|x_{n+1}-x_n\|^2 \leq h(x_n,x_{n-1}),
\end{equation}
where $\hat c:= \frac{1}{\lambda_{\max} }- L \bar\beta -\frac{\bar\gamma}{\lambda_{\min}}-\frac{L}{2} >0$. On the other hand,
\begin{equation*}
\nabla h(x_n, x_{n-1}) = (\nabla f(x_{n}) + (L \bar\beta +\frac{\bar\gamma}{ \lambda_{\min} })(x_n -x_{n-1}),  -(L \bar\beta +\frac{\bar\gamma}{ \lambda_{\min} })(x_n -x_{n-1})).
\end{equation*}
Combining this with \eqref{eq:d2}, we obtain that
\begin{equation*}
\|\nabla h(x_n, x_{n-1})\| \leq  \frac{1}{\lambda_{\min}}\|x_{n+1}-x_n\|  + 3\left( \frac{\bar\gamma}{\lambda_{\min}} +  L \bar\beta\right)\|x_n - x_{n-1}\|,
\end{equation*}
which together with \eqref{eq:cvKL2} and Lemma~\ref{l:KL} completes the proof. 
\end{proof}

Regarding the last conclusion of Theorem~\ref{t:fullseq}, we note that $\lim_{k \to +\infty} \frac{\Gamma \zeta^k}{ 1/k^2} =0$ as $\zeta \in (0,1)$. Therefore, $|f(x_k)-f(x^*)| \leq \Gamma \zeta^k$ implies $|f(x_k)-f(x^*)| \leq o(\frac{1}{k^2})$.

\begin{remark} 
In Theorem~\ref{t:fullseq}, we impose the assumption that $g$ or $h$ satisfies the KL property. This assumption is automatically satisfied when $f$ is a proper lower semicontinuous definable function (see, e.g., \cite[Section~4.3]{ABRS10}). In particular, if $f$ is a proper lower semicontinuous semialgebraic function, then so are $g$ and $h$, in which case $f$, $g$, and $h$ are KL functions with exponent in $[0, 1)$. On the other hand, if $f$ satisfies the KL property with exponent $\theta \in [\frac{1}{2}, 1)$, according to \cite[Proposition~4.1]{Yan21} and \cite[Theorem~3.6]{LP18}, $g$ and $h$ satisfy the KL property with exponent $\theta$.
\end{remark}

We require in Theorem~\ref{t:fullseq} that either $N=0$ or $\frac{1}{\lambda_{\max}} - L \beta -\frac{\gamma}{\lambda_{\min}} -\frac{L}{2} >0$. Nevertheless, without requiring these conditions, in the case when $\bar\beta =\bar\gamma = 0$, we still have linear convergence of the objective function values in the following proposition.

\begin{proposition}[Linear convergence]
Suppose that $\mathcal{H}$ is finite dimensional, that $f$ is a differentiable function whose gradient is $L$-Lipschitz continuous, and that $f$ is bounded below. Let $(x_n)_{n\in \mathbb{N}}$ be the sequence generated by AGDNL algorithm with $\lambda_{\min} >0$. Suppose further that the set $\{x \in \mathcal{H}: f(x) \leq f(x_0) + \sigma_{-1}\|w_{-1}\|^2\}$ is bounded, that $\bar\beta =\bar\gamma = 0$, and that $f$ is a KL function with exponent $\theta \leq \frac{1}{2}$. Then, there exist $\Gamma \in \R_{++}$ and $\rho \in (0,1)$ such that, for all $n \in \N$, $f(x_n) -f(\bar x) \leq \Gamma{\rho}^n$, where $\bar x $ is a cluster point of $(x_n)_{n \in \N}$.
\end{proposition}
\begin{proof}
Since $\bar\beta = \bar\gamma =0$, we have for all $n \in \N$, $\beta_n = \gamma_n =0$. Using the definitions of $\rho_n$ and $V$ in Step~\ref{step:main}\ref{step:main_line} and \eqref{eq:V}, we have $\rho_n=0$ and $V(x_{n+1}, x_n,0, \rho_n, \sigma_n) = f(x_{n+1})$. According to Theorem~\ref{t:subseq}\ref{t:subseq_simplify}, there exists $N_0 \in \mathbb{N}$ such 
that, for all $n \geq N_0$, $x_{n+1}=x_{n}-\lambda_n \nabla f(x_n)$ and $f(x_{n+1})+c \|x_{n+1}-x_n\|^2 + d \|\nabla f(x_n)\|^2 \leq \max_{[n-N]_+ \leq k \leq n}f(x_k)$,
which implies that
\begin{equation}\label{eq:cri}
f(x_{n+1})+(c+\frac{d}{\lambda_{\max}^2}) \|x_{n+1}-x_n\|^2 \leq \max_{[n-N]_+ \leq k \leq n}f(x_k).
\end{equation}
We have $c+\frac{d}{\lambda_{\max}^2} >0$ since $\lambda_{\min} >0$. Set $r(n) =m((N+1)n) $ for $n \in \N$. Since $n-N \leq m(n) \leq n$, we have $r(n) = m((N+1)n) \geq (N+1)n -N$, and so $r(n) -1 \geq(N+1)n -N -1 = (N+1)(n-1) \geq m((N+1)(n-1)) = r(n-1)$.
Moreover,
\begin{align}\label{eq:rn}
f(x_{r(n)})& \leq f(x_{m(r(n)-1)} ) -(c+\frac{d}{\lambda_{\max}^2}) \|x_{r(n)} -x_{r(n)-1} \|^2 \notag\\
& \leq f(x_{m((N+1)(n-1))} ) -(c+\frac{d}{\lambda_{\max}^2}) \|x_{r(n)} -x_{r(n)-1} \|^2 \notag\\
&= f(x_{r(n-1)})-(c+\frac{d}{\lambda_{\max}^2}) \|x_{r(n)} -x_{r(n)-1} \|^2,
\end{align}
where the first inequality follows from \eqref{eq:cri}, the second inequality follows from the fact that $(f(x_{m(n)}))_n$ is non-increasing and $r(n)-1 \geq (N+1)(n-1)$. Now set $\tilde{z}_n = x_{r(n)}$ and $\Delta_{n-1} = \|x_{r(n)} -x_{r(n)-1}\|$. It follows from \eqref{eq:rn} that
\begin{equation*}
f(\tilde{z}_n) + (c+\frac{d}{\lambda_{\max}^2})  \Delta_{n-1}^2 \leq f(\tilde{z}_{n-1}).
\end{equation*}
On the other hand,
\begin{align*}
\|\nabla f(\tilde{z}_n)\|=\|\nabla f(x_{r(n)}) \|& \leq \|\nabla f(x_{r(n)}) - \nabla f(x_{r(n)-1})\| + \| \nabla f(x_{r(n)-1})\| \notag\\
& \leq L\|x_{r(n)} -x_{r(n)-1}\| +\frac{1}{\lambda_{r(n)-1}}\|x_{r(n)} -x_{r(n)-1}\| \notag\\
& \leq (L +\frac{1}{\lambda_{\min}}) \Delta_{n-1}.
\end{align*}
According to Theorem~\ref{t:subseq}\ref{t:subseq_bounded}, $(x_n)_{n \in \N}$ is bounded, which implies that $(\tilde{z}_n)_{n \in \N}$ is bounded. 
Since $f$ satisfies the KL property with exponent $\theta \leq \frac{1}{2}$, Lemma~\ref{l:KL} yields that there exist $\Gamma \in \R_{++}$ and $\omega \in (0,1)$ such that, for all $j \in \N$,
\begin{equation}\label{eq:fz}
f(\tilde{z}_j) - f(\bar x) \leq \Gamma \omega^j,
\end{equation}
where $\bar x$ is a cluster point of $(x_n)_{n \in \N}$. Finally, for all $n \in \N$, there exists $j \in \N$, such that $[(N+1)j -N]_{+} \leq n \leq (N+1)j$. Then
\begin{equation*}
f(x_n) \leq  \max_{[(N+1)j -N]_{+} \leq k \leq (N+1)j} f(x_k) =f(x_{m((N+1)j)}) = f(x_{r(j)})=f(\tilde{z}_j)
\end{equation*}
This together with \eqref{eq:fz} implies that $f(x_n) -f(\bar x) \leq  \Gamma \omega^j = \Gamma ( \sqrt[N+1]{\omega})^{j(N+1)} \leq \Gamma{\rho}^n$, where $\rho =\sqrt[N+1]{\omega}$.
\end{proof}

\section{Special choices of parameters}
\label{s:comparison}

We now discuss some special cases of parameters such that the line search criterion~\eqref{eq:linesearch} automatically holds and we show that with suitable choices of parameters, our algorithm reduces to several existing algorithms.

\subsection{Fast gradient algorithm with only dry-like friction}
\label{r:ls}
It is worthwhile mentioning here some instances of AGDNL algorithm without line search step.
\begin{enumerate}
\item Regarding AGDNL algorithm, let $\mu \in \mathbb{R}_{++}$, $\alpha_n \equiv \bar\alpha \in (0, \frac{1}{\mu})$, $\bar\beta =\beta_{-1} \in \R_{+}$, $\beta_n \in [0, \beta_{n-1}]$, $\bar\gamma =\gamma_{-1} \in \R_{+}$, $\gamma_n \in [0,\gamma_{n-1}]$, $L_{\min} =L_{\max}=L$, $\lambda_n \equiv \lambda_{\min}=\lambda \in \R_{+}$, $\kappa_n \equiv \kappa_{\min} =\kappa \in \R_{++}$. 
Suppose that
\begin{enumerate}
\item\label{a:barc}
$\tilde c:=(1 -\mu \bar\alpha) (1 -2 \bar \gamma ) -  \frac{2 L \bar\beta +\bar\gamma + L}{2} \left(2\lambda(1-\mu \bar\alpha)+ \kappa\mu\right)  \geq 0$; 
\item \label{a:bard}
$\tilde d:=2(1- \mu \bar\alpha)(\lambda^2 -\bar\gamma \lambda)-\mu \bar\gamma \kappa+2\mu\lambda\kappa  \geq  0$.
\end{enumerate}
Let $c \in [0,  \frac{\tilde c}{\mu^2 \underline{\nu}}] $, where $\underline{\nu}:=\frac{2 \lambda(1 -\mu \bar\alpha)}{\mu^2}+ \frac{\kappa}{\mu}$.
Let $d \in [0, \frac{\tilde d}{2\mu^2 \underline{\nu}}]$ such that if $\lambda >0$ and $c=0$, then $d>0$.
Suppose that $\bar \gamma=0$ if $c=0$. Then the line search criterion~\eqref{eq:linesearch} holds trivially for all $n \in \N$ and AGDNL algorithm reduces to the following algorithm without line search
\begin{equation*}
\small
\begin{cases}
u_n &= x_n + \beta_n(x_n - x_{n-1}),\\ 
v_n & = x_n + \gamma_n (x_n -x_{n-1}),\\
w_n &\in \prox_{\tau_n \phi}\left(\bar\alpha(x_n -v_{n-1}+\lambda \nabla f(u_{n-1}))  -\theta_n \frac{\nabla f(u_n)}{\|\nabla f(u_n)\|} )-\kappa_n \nabla f(u_n)+ e_n\right ),\\
x_{n+1}\!\!\! &= v_n- \lambda \nabla f(u_n) +\mu w_n,
\end{cases}
\end{equation*}
which we will name as the \emph{fast algorithm with dry-like friction (AGD)}. Indeed, it follows from assumption~\ref{a:barc} that
\begin{equation*}
(1-\mu \bar \alpha) - L\left(\lambda(1-\mu \bar\alpha) + \frac{ \kappa \mu}{2}\right) \geq  (1-\mu \bar\alpha) 2 \bar\gamma  +\frac{2 L \bar\beta +\bar\gamma }{2} \left(2\lambda(1-\mu \bar\alpha)+ \mu \kappa\right) \geq 0,
\end{equation*}
which implies that $\frac{1}{L} > \lambda =\lambda_{\min}$ and  $\bar \kappa = \frac{2(1-\mu \bar\alpha )}{\mu L} -\frac{2(1-\mu \bar\alpha)\lambda}{\mu} \geq \kappa = \kappa_{\min}$. It also can be seen that $c \in [0, \bar c]$ and $d \in [0, \bar d]$. If $c =\bar c$, then $\frac{\tilde c}{\mu^2 \underline{\nu}} = \bar c$, and so $2(1 -\mu \bar \alpha)- \mu^2 L \nu - 2\tilde c  =0$, i.e., $4(1-\mu \bar\alpha)\bar\gamma  +(2 L \bar\beta +\bar\gamma ) \left(2\lambda(1-\mu \bar\alpha)+ \mu \kappa\right) =0$, which yields $\bar \beta =0$ and $\bar \gamma =0$. Similarly, if $d =\bar d$, then $\bar \gamma=0$. Therefore, the conditions of all the parameters in Step~\ref{step:init} and Step~\ref{step:main} hold. On the other hand, by the definition of $c_n$ and $d_n$ in \eqref{d:cn} and \eqref{d:dn}, $c_n \geq \frac{\tilde c}{\mu^2}$ and $d_n \geq \frac{\tilde d}{\mu^2}$. Since $L_{\min}=L_{\max}=L$ and $\beta_n \in [0, \beta_{n-1}]$, we have for all $n \in \N$ that $L_{n} =L_{\max}$, and so $L_n\beta_n\leq L_{n-1}\beta_{n-1}$. Moreover, for all $n \in \N$, $\nu_{n} = \frac{2 \lambda_{\min}(1 -\mu \bar\alpha)}{\mu^2}+ \frac{\kappa_{\min} }{\mu} = \underline{\nu}$. Now, using Lemma~\ref{l:descent} and noting that $d \in [0, \frac{\tilde d}{\mu^2 \underline{\nu}}] $ and $c \in [0,  \frac{\tilde c}{\mu^2 \underline{\nu}}]$, we obtain 
\begin{align*}
  V(z_{n+1}, x_n, w_n,\rho_n, \sigma_n) +c\|z_{n+1} -x_n||^2 + d\|\nabla f(u_n)\|^2 + \frac{\delta}{\underline{\nu}} \| w_n\| \notag\\
   \leq V(x_n, x_{n-1}, w_{n-1},\rho_{n-1}, \sigma_{n-1}),
\end{align*} 
and the line search criterion \eqref{eq:linesearch} is satisfied. On the other hand, in this case, it follows from $\tilde c \geq 0$ that $ 2 (1 -\mu \bar \alpha) \lambda_n + \kappa_n \mu < \frac{2(1-\mu \bar \alpha)}{L}$. However, for AGDNL, we can have $ 2 (1 -\mu \bar \alpha) \lambda_n + \kappa_n \mu \geq \frac{2(1-\mu \bar \alpha)}{L}$. These choices may
help AGDNL speed up the convergence.

\item 
Let $\gamma_n \equiv 0$, $\lambda_n \equiv 0$, $\kappa_n \equiv \kappa \in \R_{++}$, $\alpha_n \equiv \bar\alpha$, $\bar \beta =\beta_{-1} \in \R_+$, and $L_{\max} = L_{\min} =L$. Let $ \beta_n \in [0, \beta_{n-1}]$. It can be seen that the assumption~\ref{a:bard} is automatically satisfied while the assumption~\ref{a:barc} reduces to 
\begin{equation}\label{a:AGD}
1 -\mu  \bar\alpha-\frac{L\mu(2\bar \beta+1) \kappa}{2} \geq 0.
\end{equation}
Let $c \in [0,  \frac{1}{\mu \kappa} (1 -\mu \bar\alpha-\frac{L\mu(2\bar\beta+1) \kappa}{2})]$. Then AGD algorithm reduces to the following algorithm
\begin{equation}\label{AGD}
\begin{cases}
u_n &= x_n + \beta_n(x_n - x_{n-1}),\\ 
w_n &\in \prox_{\tau_n \phi}\left(\bar\alpha(x_n -x_{n-1})  -\theta_n \frac{\nabla f(u_n)}{\|\nabla f(u_n)\|} -\kappa\nabla f(u_n)+ e_n\right ),\\
 x_{n+1}\!\!\! &= x_n +\mu w_n.
\end{cases}
\end{equation}
\end{enumerate} 

\subsection{Comparison with existing algorithms}
The general framework of AGDNL algorithm covers various algorithms in the literature as we will see below.
\begin{enumerate}
\item 
 Let $\beta_n \equiv 0$, $\theta_n \equiv 0$, $e_n \equiv 0$, and $\gamma_n \equiv 0$. Let  $\beta$, $h$ and $\gamma$ are positive parameters and let $\lambda_n \equiv \beta h$, $\mu =h$, $\tau_n \equiv \frac{h}{1+\gamma h}$, $\alpha_n \equiv \frac{1}{h(1+\gamma h)} $, and $\kappa_n \equiv \frac{(1-\gamma \beta)h}{1 +\gamma h}$. Suppose that the assumption in \cite[Theorem 2.1]{AAL21} holds, i.e., $hL \leq \frac{2 \gamma}{\gamma \beta +1}$. Then $\tilde c \geq 0$ and $\tilde d >0$. By Remark~\ref{r:ls}, 
 criterion~\eqref{eq:linesearch} holds trivially for all $n \in \mathbb{N}$ and AGDNL algorithm reduces to the following algorithm 
 \begin{align*}
 x_{n+1} &= x_n -\beta h \nabla f(x_n) \notag\\
 & + h \prox_{\frac{h}{1 +h\gamma}\varphi} \left( \tfrac{1}{h(1+\gamma h)} (x_n -x_{n-1}) +\tfrac{\beta}{1+\gamma h} \nabla f(x_{n-1}) + \tfrac{(\gamma \beta -1)h}{1+\gamma h} \nabla f(x_n)\right),
 \end{align*}
which can be written as
\begin{equation*}
\begin{cases} 
z_n &= \frac{1}{h}(x_n - x_{n-1}) + \beta \nabla f(x_{n-1})\\ 
x_{n+1}\!\!\! &=x_n- \beta h \nabla f(x_n) + h \prox_{\frac{h}{1+\gamma h} \phi} \left( \frac{1}{1+\gamma h} z_n + \frac{(\gamma \beta -1)h}{1 + \gamma h} \nabla f(x_n)\right).
\end{cases}
\end{equation*}
This is IPAHDD-C1 algorithm and so  we can get \cite[Theorem 2.1]{AAL21} by using Theorem~\ref{t:subseq}.

\item Let $h$ and $\gamma$ are positive. Let $\mu = h$, $\lambda_n \equiv 0$, $\gamma_n \equiv 0$, $\beta_n \equiv 0$, $\theta_n \equiv 0$, $\tau_n \equiv \frac{h}{1+ h\gamma}$, $\alpha_n \equiv \frac{1}{h(1+h\gamma)}$, $\kappa_n \equiv \frac{h}{1+ h\gamma}=\kappa_{\min} =\kappa_{\max}$, and $e_n \equiv 0$. Then the algorithm given in \eqref{AGD}, that is a special case of AGDNL, becomes
\begin{equation*}\label{eq:IPGDF}
x_{n+1}=x_n+h\prox_{ \frac{h}{1+ h\gamma}\phi}\left(\frac{1}{h(1+h\gamma)}(x_n -x_{n-1}) -\frac{h}{1+ h\gamma} \nabla f(x_n) \right ),
\end{equation*}
which is IPGDF algorithm in \cite{AA22}. On the other hand, assumption \eqref{a:AGD} becomes $1 - \frac{1}{1 + h\gamma} - \frac{L h^2}{2(1+h \gamma)} \geq 0$, or equivalently, $h \leq \frac{2\gamma}{L}$, which is the assumption in \cite[Theorem 1]{AA22}.  

  Let $\delta =r \kappa_{\min}$. Using Proposition~\ref{p:l0N0}, $x_n \to x_{\infty}$ as $n \to +\infty$ and $-\nabla f(x_{\infty}) \in \phi(0)$. Suppose that  $\|\nabla f(x_{\infty})\| < r$. Then $\|\nabla f(x_{\infty})\| < \frac{\delta}{\kappa_{\min}} = \frac{\delta}{\kappa_{\max}}$. According to Proposition~\ref{p:l0N0}, there exists $N_0 \in \mathbb{N}$ such that for all $n \geq N_0$, $x_n = x_{\infty}$, i.e., $(x_n)_{n \in \N}$ is finitely convergent. Thus, \cite[Theorem~1 and Theorem~2]{AA22} are instances of Proposition~\ref{p:l0N0}.

\item  
Let $h$ and $\gamma$ are positive. Let $\mu = h$, $\lambda_n \equiv 0$, $\gamma_n \equiv 0$, $\beta_n  \equiv \frac{1}{1+h\gamma}$, $\theta_n  \equiv 0$, $\tau_n \equiv  \frac{h}{1+ h\gamma}$, $\alpha_n \equiv \frac{1}{h(1+h\gamma)}$, $\kappa_n \equiv \frac{h}{1+ h\gamma}$, and $e_n \equiv 0$. Then the algorithm given in \eqref{AGD} reduces to
\begin{equation*}	
\begin{cases}
u_n &= x_n + \frac{1}{1+h\gamma}(x_n - x_{n-1})\\
x_{n+1}\!\!\! &=x_n+h\prox_{ \frac{h}{1+ h\gamma}\phi}\left(\frac{1}{h}(u_n -x_n) -\frac{h}{1+ h\gamma} \nabla f(u_n) \right ).
\end{cases}
\end{equation*}
It is IPGDF-NF algorithm in \cite[Theorem 4]{AA22}. Moreover, assumption \eqref{a:AGD} reduces to $1 - \frac{1}{1 + h\gamma} - \frac{L h^2}{2(1+h \gamma)} - \frac{L h^2}{(1+h \gamma)^2} \geq  0$.
Now, we suppose that $h$ and $\gamma$ satisfy the assumption in \cite[Theorem 4]{AA22}, i.e., $h < \frac{2\gamma}{3L}$. Then it can be  that $1 - \frac{1}{1 + h\gamma} - \frac{L h^2}{2(1+h \gamma)} - \frac{L h^2}{(1+h \gamma)^2}
>\frac{3Lh^2}{2(1 + h\gamma)}- \frac{Lh^2}{1 + h\gamma}\left(\frac{1}{2} + \frac{1}{1 +h\gamma}\right) \geq 0$. Thus, $h$ and $\gamma$ satisfy assumption \eqref{a:AGD}, and so \cite[Theorem 4]{AA22} can be obtained by using Proposition~\ref{p:l0N0}.
\end{enumerate}

\section{Adaptation to nonsmooth case}
\label{s:nonsmooth}

In this section, $f\colon \mathcal{H} \to \left(-\infty, +\infty\right]$ is a proper lower semicontinuous function that is possibly nonconvex and nonsmooth. The following lemma collects some important variational properties of the proximal operator and Moreau envelope of $\alpha$-weakly convex functions. From now on, we adopt a convention that $\frac{1}{\alpha} = +\infty$ if $\alpha =0$.

\begin{lemma}[Proximal operator and Moreau envelope of weakly convex functions] 
\label{l:weakcvx}
Suppose that $f\colon \mathcal{H}\to (-\infty, +\infty]$ is a proper lower semicontinuous $\alpha$-weakly convex function and let $\xi \in (0, \frac{1}{\alpha})$. Then the following hold:
\begin{enumerate}
\item\label{l:weakcvx_Lips} 
$\prox_{\xi f}$ is single-valued and $\frac{1}{1 -\xi \alpha}$-Lipschitz continuous on $\mathcal{H}$.
\item\label{l:weakcvx_diff} 
$f_{\xi}$ is continuously differentiable on $\mathcal{H}$ with gradient $\nabla f_{\xi} = \frac{1}{\xi}(\Id -\prox_{\xi f})$ being $L$-Lipschitz continuous, where $\Id$ is the identity mapping and $L=\max\{\frac{\alpha}{1 - \xi \alpha}, \frac{1}{\xi}\}$.
\item\label{l:weakcvx_crit}  
The functions $f$ and $f_{\xi}$ have the same critical points and critical values. 
\end{enumerate}
\end{lemma}
\begin{proof}
\ref{l:weakcvx_Lips} and \ref{l:weakcvx_diff} follow from \cite[Theorem~3.4]{AA93}, while \ref{l:weakcvx_crit} follows from \cite[Proposition~3.6]{AA93}.
\end{proof}

Exploiting the properties in Lemma~\ref{l:weakcvx}, we can solve weakly convex nonsmooth optimization problems by applying our AGDNL algorithm to the Moreau envelope of the objective function.
 
\begin{theorem}[AGDNL for nonsmooth case]
\label{t:DFGp}
Suppose that $f\colon \mathcal{H}\to (-\infty, +\infty]$ is a proper lower semicontinuous $\alpha$-weakly convex function. Let $\xi \in (0, \frac{1}{\alpha})$ and let $(x_n)_{n\in \mathbb{N}}$ be the sequence generated by AGDLN algorithm applied to $f_{\xi}$. Suppose also that $f$ is bounded below and that the set $\{x \in \mathcal{H}: f_{\xi}(x) \leq f_{\xi}(x_0) +\sigma_{-1}\|w_{-1}\|^2 \}$ is bounded. Then the following hold:
\begin{enumerate}
\item $(x_n)_{n \in \N}$ is bounded, $\lim_{n \to +\infty}\|x_{n+1}-x_n\| = 0$, and $\lim_{n \to +\infty}\|w_n\| = 0$.
\item If $\lambda_{\min} >0$, then every strong cluster point of $(x_n)_{n \in \N}$ is a critical point of $f$.
\item Suppose that $\mathcal{H}$ is finite dimensional, that $\lambda_{\min} >0$, and that one of the following holds:
\begin{enumerate}
\item \label{ADFG_N:i}  $N=0$ and $ g^{(\xi)}(x,y,s):= f_{\xi}(x)+ \rho \|x -y\|^2 $ satisfies the KL property at $(\bar x, \bar x, \rho)$ for all $\bar x \in \mathcal{H}$ and $\rho \in [0, \bar\rho]$, where $\bar\rho =\frac{1}{\mu^2}(1 -\mu \bar\alpha)\frac{\bar\gamma}{\underline{\nu}} + \frac{ L_{\max}\bar\beta +\bar\gamma}{2 }$.
\item \label{ADFG_N:ii} $\frac{1}{\lambda_{\max} }- L \bar\beta -\frac{\bar\gamma}{\lambda_{\min}}-\frac{L}{2} >0$ and $h^{(\xi)}(x,y): =f_{\xi}(x) +(\frac{L \bar\beta}{2} +\frac{\bar\gamma}{2 \lambda_{\min} }) \|x-y\|^2$ satisfies the KL property at $(\bar x, \bar x)$  for all $\bar x \in \mathcal{H}$.
\end{enumerate}
 Then $\sum_{k=0}^{+\infty} \|x_{n+1} -x_n\| < +\infty$ and $(x_n)_{n\in \N}$ converges to a critical point $x^*$ of $f$. Moreover, if the KL property in \ref{ADFG_N:i} or \ref{ADFG_N:ii} is satisfied with exponent $\theta \leq 1/2$, then there exist $\Gamma \in \R_{++}$ and $\zeta \in (0,1)$ such that, for all $n \in \N$, $\|x_n -x^*\| \leq \Gamma\zeta^{n/2}$.
\end{enumerate}
\end{theorem}
\begin{proof}
Since $f$ is bounded below, the definition of $f_{\xi}$ implies that $f_{\xi}$ is also bounded below. The conclusion now follows from Theorem~\ref{t:fullseq} with noting from Lemma~\ref{l:weakcvx}\ref{l:weakcvx_crit} that $f$ and $f_{\xi}$ have the same critical points. 
\end{proof}

In Theorem~\ref{t:DFGp}, the KL property of $g^{(\xi)}$ and $h^{(\xi)}$ plays an important role for the convergence analysis of the whole sequence generated by AGDNL algorithm. According to \cite[Proposition~4.1]{Yan21} and \cite[Theorem~3.6]{LP18}, if $f_{\xi}$ satisfies the KL property with exponent $\theta \in [\frac{1}{2},1)$, then $g^{(\xi)}$ and $h^{(\xi)}$ satisfy the KL property with exponent $\theta$. In the sequel, we will provide a sufficient condition for a proper lower semicontinuous function $f$ such that $f_{\xi}$ satisfies the KL property. In order to do this, we first establish a result on the KL property for marginal functions, for which we recall that a set-valued mapping $P\colon \mathcal{H} \rightrightarrows \mathcal{H}$ is \emph{inner semicontinuous} at $(x,y)$ with $y \in P(x)$ if for every sequence $x_n \to  x$, there is a sequence 
$y_n \to y$ with $y_n \in P(x_n)$.

\begin{lemma}[KL property for marginal functions]
\label{l:margin}
Let $\mathcal{K}$ be a real Hilbert space, let $F: \mathcal{H} \times \mathcal{K} \to (-\infty, +\infty]$ be a proper lower semicontinuous function, and let $C$ be a nonempty set in $\mathcal{K}$. For each $x \in \mathcal{H}$, set $h(x) =\inf_{y\in C} F(x,y)$ and $P(x)= \{y \in C: F(x,y) = h(x)\}$.
\begin{enumerate}
\item\label{l:margin_incl}
Let $x \in \mathcal{H}$ and suppose that $h$ is finite at $x$, and that $P$ is inner semicontinuous at $(x,y_{x})$ for some $y_{x} \in P(x)\cap \inte C$. Then $\partial h(x) \subseteq \left\{x^*: (x^*,0) \in \partial F(x,y_x)\right\}$.
    
\item\label{l:margin_KL}
Let $\bar x \in \mathcal{H}$ and suppose that there exists $\varepsilon \in \mathbb{R}_{++}$ such that for all $x \in B(\bar x; \varepsilon)$, $h$ is finite at $x$ and $P$ is inner semicontinuous at $(x,y_{x})$ for some $y_{x} \in P(x)\cap \inte C$. Suppose further that there exist $\bar y \in P(\bar x)$, $c \in \mathbb{R}_{++}$, $\theta \in \mathbb{R}_{++}$, and $\nu \in (0,+\infty]$ such that
\begin{equation}\label{eq:KLF}
\dist(0,\partial F(x,y_x)) \geq c(F(x,y_x) -F(\bar x,\bar y))^{\theta}
\end{equation}
whenever $x \in B(\bar x; \varepsilon)$ with $F(\bar x,\bar y) < F(x,y_x) < F(\bar x,\bar y) +\nu$. Then $h$ satisfies the KL property at $\bar x$ with exponent $\theta$.
\end{enumerate}
\end{lemma}
\begin{proof}
\ref{l:margin_incl}: Let $G(x,y) =F(x,y) +\delta_C(y) =F(x,y) +\delta_{\mathcal{H}\times C}(x,y)$. Since $y_x \in \inte C$, we have that $\delta_{\mathcal{H}\times C}$ is Lipschitz continuous around $(x, y_x)$ and $\partial \delta_{\mathcal{H}\times C}(x,y_x) =\{0\}$. Noting that $F(x, y_x) =h(x) < +\infty$ and using \cite[Theorem~3.36]{Mor06} along with its following remark, we obtain that $\partial G(x,y_x) \subseteq \partial F(x,y_x) +\partial \delta_{\mathcal{H}\times C}(x,y_x) =\partial F(x,y_x)$. Now, by \cite[Theorem~1.108(i)]{Mor06}, $\partial h(x) \subseteq \left\{x^*: (x^*,0) \in \partial G(x,y_x)\right\} \subseteq \left\{x^*: (x^*,0) \in \partial F(x,y_x)\right\}$.

\ref{l:margin_KL}: 
It follows from \ref{l:margin_incl} that for all $x \in B(\bar x; \varepsilon)$,
\begin{equation}\label{eq:inclusion}
\partial h(x) \subseteq \left\{x^*: (x^*,0) \in \partial F(x,y_x)\right\}. 
\end{equation}
Now, let $x\in B(\bar x;\varepsilon)$ with $h(\bar{x}) <h(x) <h(\bar{x}) +\nu$. We note from $\bar y \in C(\bar x)$ and $y_x \in P(x)$ that $h(\bar{x}) =F(\bar{x},\bar{y})$ and $h(x) =F(x,y_x)$, and so $F(\bar{x},\bar{y}) < F(x,y_x) < F(\bar{x},\bar{y}) +\nu$. Therefore, we derive from \eqref{eq:inclusion} that
\begin{align*}
\dist (0,\partial h(x)) &\geq \inf \left\{\|x^*\| :(x^*,0) \in \partial F(x,y_x) \right\} \\
&\geq \dist(0, \partial F(x,y_x)) \\
&\geq c (F(x,y_x) -F(\bar x,\bar y))^{\theta} = c (  h(x) -h(\bar x))^{\theta},
\end{align*}
where the last inequality follows from \eqref{eq:KLF}. The proof is complete.
\end{proof}

The following result can be found in \cite[Remark~3.2(b)]{ABRS10} and \cite[Lemma~2.1]{LP18} with a finite-dimensional setting. For self-containedness, we include here a proof which follows the argument in the proof of \cite[Lemma~2.1]{ABRS10}.

\begin{lemma}[KL property at noncritical points]
\label{l:noncrit}
Let $h\colon \mathcal{H}\to (-\infty,+\infty]$ be a proper lower semicontinuous function and let $\bar x \in \dom \partial h$. Suppose that $0 \notin \partial h(\bar x)$. Then $h$ satisfies the KL property at $\bar x$ with exponent $\theta \in [0,1)$.
\end{lemma}
\begin{proof}
It suffices to prove that there exists $c \in \mathbb{R}_{++}$ such that $\dist(0, \partial h(x)) \geq c$ whenever $\|x -\bar x\| \leq c$ and $|h(x) -h(\bar x)| \leq c$. Suppose on the contrary that this is not true. We then find a sequence $(x_n)_{n \in \mathbb{N}}$ such that $\|x_n -\bar x\| \leq 1/n$ and $|h(x_n) -h(\bar x)| \leq 1/n$ but $\dist(0, \partial h(x_n)) < 1/n$. It follows that $x_n \to \bar x$ and $h(x_n) \to h(\bar x)$ as $n\to +\infty$ and that, for each $n\in \mathbb{N}$, there exists $u_n \in \partial h(x_n)$ with $\|u_n\| < 1/n$. Letting $n\to +\infty$, we drive that $0 \in \partial h(\bar x)$, which contradicts the assumption.
\end{proof}

\begin{theorem}[Exponent for Moreau envelope of KL functions]
\label{t:EnvKL}
Suppose that $f\colon \mathcal{H}\to (-\infty, +\infty]$ is a proper lower semicontinuous function and let $\xi \in \mathbb{R}_{++}$. 
\begin{enumerate}
\item\label{t:EnvKL_point} 
Let $\bar x \in \dom \partial f\cap \dom \partial f_\xi$. Suppose that $f$ satisfies the KL property at $\bar x$ with exponent $\theta \in [\frac{1}{2},1)$ and that if $0 \in \partial f_{\xi}(\bar x)$, then $\prox_{\xi f}(\bar x) = \bar x$ and $\prox_{\xi f}$ is single-valued and Lipschitz continuous around $\bar x$. Then $f_{\xi}$ satisfies the KL property at $\bar{x}$ with exponent $\theta$.
\item\label{t:EnvKL_full} 
Suppose that $f$ is KL function with exponent $\theta \in [\frac{1}{2},1)$ and that for every critical point $\bar x$ of $f_{\xi}$, $\prox_{\xi f}(\bar x) = \bar x$ and $\prox_{\xi f}$ is single-valued and Lipschitz continuous around $\bar x$. Then $f_{\xi}$ is a KL function with exponent $\theta$.
\end{enumerate}
\end{theorem}
\begin{proof}
We first claim that, for all $x \in \mathcal{H}$, 
\begin{equation}\label{eq:ranprox}
\prox_{\xi f}(x) \subseteq \dom \partial f.    
\end{equation}
Indeed, we can assume that $\prox_{\xi f}(x) \neq \varnothing$ as the claim is trivial otherwise. Let $p \in \prox_{\xi f}(x)$. Then $p$ is a minimizer of $f +\frac{1}{2\xi}\|\cdot-x\|^2$, which yields $0 \in \partial f(p) +\frac{1}{\xi}(p-x)$, and so $p \in \dom \partial f$. This confirms \eqref{eq:ranprox}.

\ref{t:EnvKL_point}: In view of Lemma~\ref{l:noncrit}, it suffices to assume that $0 \in \partial f_{\xi}(\bar x)$. Then $\prox_{\xi f}(\bar x) = \bar x$ and there exist $\bar \varepsilon \in \mathbb{R}_{++}$ and $K \in [1,+\infty)$ such that $\prox_{\xi f}$ is single-valued and $K$-Lipschitz continuous on $B(\bar x; \bar \varepsilon)$. As $f$ satisfies the KL property at $\bar{x}$ with exponent $\theta \in [\frac{1}{2},1)$, by applying \cite[Theorem 3.6]{LP18} (whose proof is still valid in Hibert spaces), $F(x,y) =f(y) +\frac{1}{2\xi}\|x-y\|^2$ satisfies the KL property at $(\bar{x},\bar{x})$ with exponent $\theta$, i.e., there exist $c \in \mathbb{R}_{++}$, $\tilde \varepsilon \in (0, \bar \varepsilon]$, and $\nu \in (0,+\infty]$ such that for all $x \in B(\bar x; \tilde \varepsilon)$, $y \in B(\bar x; \tilde \varepsilon) \cap \dom \partial f$, and $F(\bar{x},\bar{x})  \leq F(x,y) < F(\bar{x},\bar{x}) +\nu$, one has
\begin{equation}\label{eq:KLF1P}
\dist(0,\partial F(x,y)) \geq c (F(x,y) -F(\bar{x},\bar{x}))^{\theta}. 
\end{equation}
Let $\varepsilon = \frac{\tilde \varepsilon}{K} \leq \tilde \varepsilon \leq \bar \varepsilon $. Then, for all $x \in B(\bar x, \varepsilon) \subseteq B(\bar x, \tilde \varepsilon)$,
\begin{equation*}
\|\prox_{\xi f}(x) -\bar x\|=    \|\prox_{\xi f}(x) -\prox_{\xi f}(\bar x)\| \leq K\|x -\bar x\|  \leq \tilde \varepsilon, 
\end{equation*}
which together with \eqref{eq:ranprox} implies that $\prox_{\xi f} (x) \in B(\bar x; \tilde \varepsilon)\cap \dom \partial f$. Thus, \eqref{eq:KLF1P} holds for all $x \in B(\bar x; \varepsilon)$, $y = \prox_{\xi f}(x)$ with $F(\bar x, \bar x) < F(x,y) < F(\bar x, \bar x)+\nu$.
By the single-valued property and continuity of $\prox_{\xi f}$ on $B(\bar x; \varepsilon)$, for all $x \in B(\bar x; \varepsilon)$, $\prox_{\xi f}$ is inner semicontinuous at $(x,y_x)$ with $y_x = \prox_{\xi f}(x)$. The conclusion now follows by applying Lemma~\ref{l:margin}\ref{l:margin_KL} with $h =f_\xi$, $P =\prox_{\xi f}$, and $C =\mathcal{H}$. 

\ref{t:EnvKL_full}: Let $\bar x \in \dom \partial f_{\xi}$. If $0 \notin \partial f_{\xi}(\bar x)$, then by Lemma~\ref{l:noncrit}, $f_{\xi}$ satisfies the KL property at $\bar x$ with exponent $\theta$. Now, assume that $0 \in \partial f_{\xi}(\bar x)$, i.e., $\bar x$ is a critical point of $f_\xi$. Then $\prox_{\xi f}(\bar x) = \bar x$ and, by \eqref{eq:ranprox}, $\bar x \in \dom \partial f$. Since $f$ satisfies the KL property at $\bar{x}$ with exponent $\theta \in [\frac{1}{2},1)$, so does $f_\xi$ due to \ref{t:EnvKL_pointW}. The proof is complete.
\end{proof}

\begin{corollary}[Exponent for Moreau envelope of weakly convex KL functions]
\label{c:EnvKL}
Suppose that $f\colon \mathcal{H}\to (-\infty, +\infty]$ is a proper lower semicontinuous $\alpha$-weakly convex function and let $\xi \in (0, \frac{1}{\alpha})$. Then the following hold:
\begin{enumerate}
\item\label{t:EnvKL_pointW}
If $f$ satisfies the KL property at $\bar{x} \in \dom \partial f$ with exponent $\theta \in [\frac{1}{2},1)$, then $f_{\xi}$ satisfies the KL property at $\bar{x}$ with exponent $\theta$.
\item\label{t:EnvKL_fullW}
If $f$ is a KL function with exponent $\theta \in [\frac{1}{2},1)$, then $f_{\xi}$ is a KL function with exponent $\theta$.
\end{enumerate}
\end{corollary}
\begin{proof}
We have from Lemma~\ref{l:weakcvx}\ref{l:weakcvx_Lips}\&\ref{l:weakcvx_diff} that $\prox_{\xi f}$ is single-valued and $\frac{1}{1 -\xi \alpha}$-Lipschitz continuous on $\mathcal{H}$ and that $f_\xi$ is differentiable on $\mathcal{H}$ with $\nabla f_{\xi} = \frac{1}{\xi}(\Id -\prox_{\xi f})$. The latter implies that $\bar x$ is a critical point of $f_\xi$ if and only if $\prox_{\xi f}(\bar x) = \bar x$. The conclusion now follows from Theorem~\ref{t:EnvKL}.
\end{proof}

\begin{remark}
It is shown in \cite[Theorem~3.4]{LP18} that if $f\colon \mathbb{R}^d \to (-\infty,+\infty]$ is a proper lower semicontinuous convex function that is continuous on $\dom \partial f$ and $f$ is a KL function with exponent $\theta \in (0, 2/3)$, then for all $\xi \in \mathbb{R}_{++}$, $f_{\xi}$ is a KL function with an exponent of $\max\{\frac{1}{2}, \frac{\theta}{2 - 2 \theta}\}$. In Corollary~\ref{c:EnvKL}, we extend this result in the case where $\theta \geq \frac{1}{2}$ to weakly convex functions $f$ in Hilbert spaces and obtain that $f_{\xi}$ is a KL function with exponent $\theta \leq \max\{\frac{1}{2}, \frac{\theta}{2 - 2 \theta}\}$ without requiring the continuity of $f$ on $\dom \partial f$.
\end{remark}

\section{Numerical simulation}
\label{s:numerical simulation}

In this section, we illustrate the efficiency of our algorithm. We first use the proposed algorithm and some existing algorithms in the literature to address the multi-class logistic regression problem. For AGDNL in all the experiments, we set $c=10^{-8}$, $N=2$, $\bar\beta =\bar\gamma =1$, $\lambda_{\min} = \frac{0.4}{L_{\max}}$, and $\kappa_{\min} = \frac{0.9}{\mu L_{\max}}$, $1/L_{\min} = L_{max} = 10^5$, $\beta = \gamma =0.9$, $\alpha = 0.01$, $\lambda_{-1} = \frac{0.4}{L_{-1}}$, $\kappa_{-1} = \frac{0.9}{L_{-1}}$, where $L_{-1} \in [L_{\min}, L_{\max}]$, and $\ell_1 = \ell_2 =\ell_3 =0.9$. 

\subsection{Multi-class logistic regression problem}

Let $\{(x_i, y_i): i = 1, \dots, n\}$ be a training set with explanatory variables $x_i \in \R^d$ and their corresponding labels $y_i \in \{1, \dots, K\}$, where $d$ is the number of features and $K$ is the number of classes. Let $W$ be the $d \times K$ matrix whose columns are denoted by $W_1, \dots, W_K$ and let $b = (b_1, \dots, b_K) \in \R^K$. 
The conditional probability of the training observations $x_i$ belonging to its correct classes $y_i$ is defined by
$p(Y = y_i|X = x_i) =\frac{\exp(b_{y_i} + W^T_{y_i }x_i)}{\sum_{j=1}^{K}\exp(b_j + W^T_j x_i)}.$ The problem is to find $(W, b)$ such that the sum of these conditional probabilities is maximized. This can be done by minimizing the negative log-likelihood function $L(W, b) = -\frac{1}{n}\sum_{i=1}^{n}\log p(Y = y_i|X = x_i)=-\frac{1}{n}\sum_{i=1}^{n}\log \left(\frac{\exp(b_{y_i} + W^T_{y_i }x_i)}{\sum_{j=1}^{K}\exp(b_j + W^T_j x_i)} \right).$
To prevent overfitting issue, we add a regularization term of squared magnitude of $W$ to the loss function. Then the regularized multi-class logistic regression problem is formulated as $\min_{W \in \mathbb{R}^{d\times K},\, b \in \mathbb{R}^K} F(W,b),$ where $F(W,b) = L(W,b) +\lambda \|W\|^2$. We note that $F$ is a differentiable function with Lipschitz continuous gradient. Moreover, $F$ is a KL function as it is definable in the log-exp structure (see \cite[Section~4.3]{ABRS10}).

\begin{figure}[!ht]
\centering
\subfloat[DNA dataset]{\label{fig:dna}\includegraphics[width=4cm]{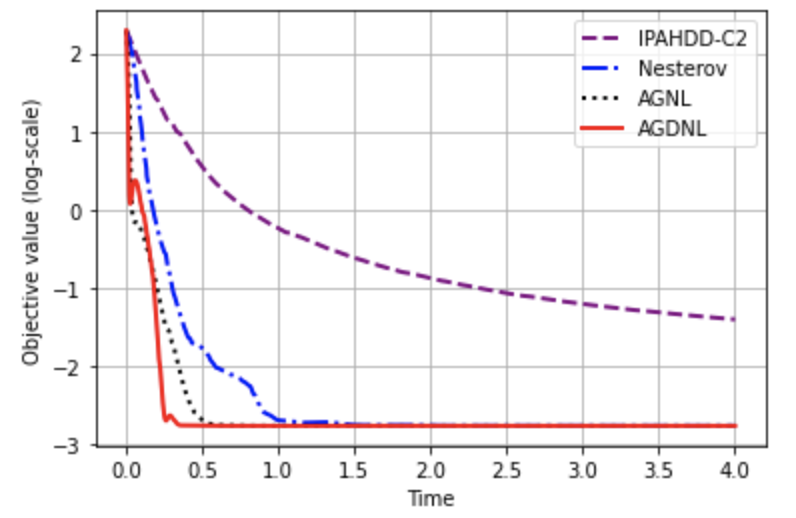}}
\subfloat[Satimage dataset]{\label{fig:satimage}\includegraphics[width=4cm]{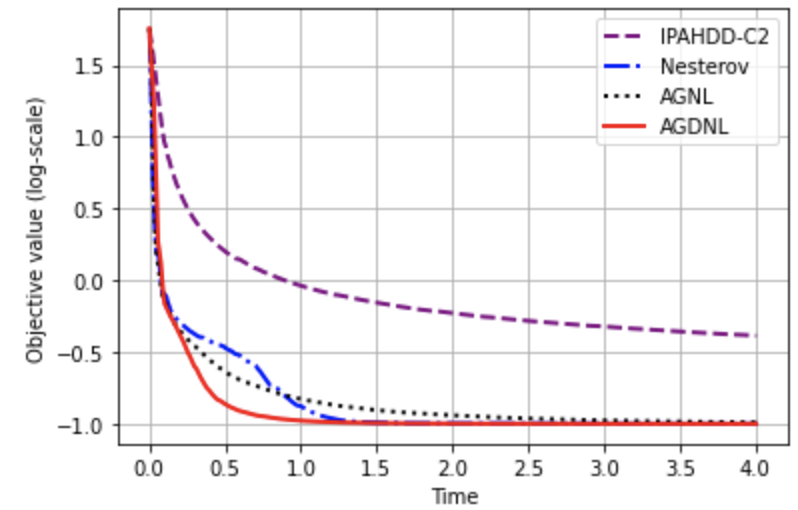}}
\subfloat[Pendigits dataset]{\label{fig:pendigits}\includegraphics[width=4cm]{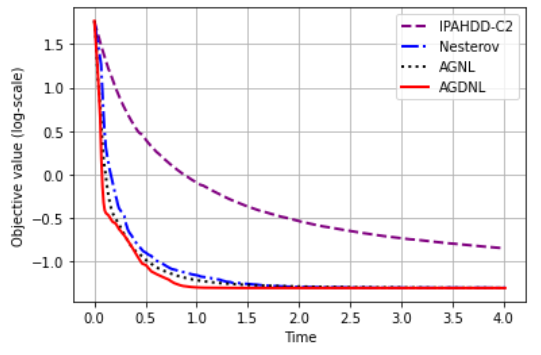}}

\subfloat[USPS dataset]{\label{fig:usps}\includegraphics[width=4cm]{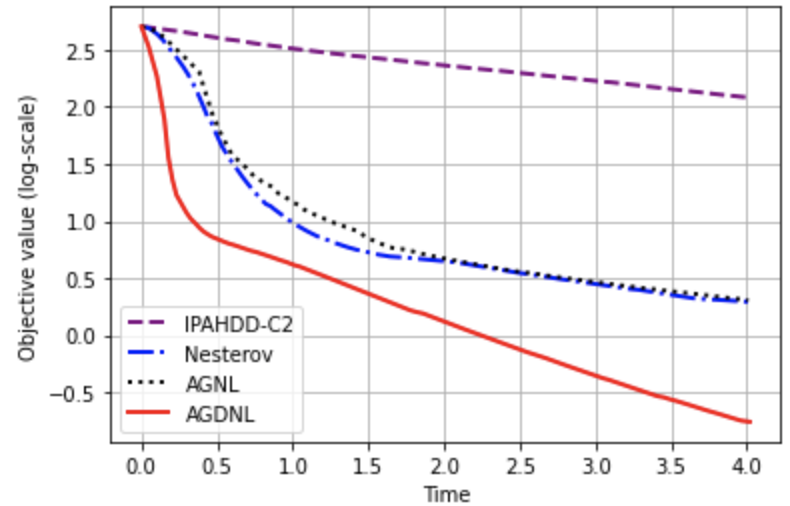}}
\subfloat[Letter dataset]{\label{fig:letter}\includegraphics[width=4cm]{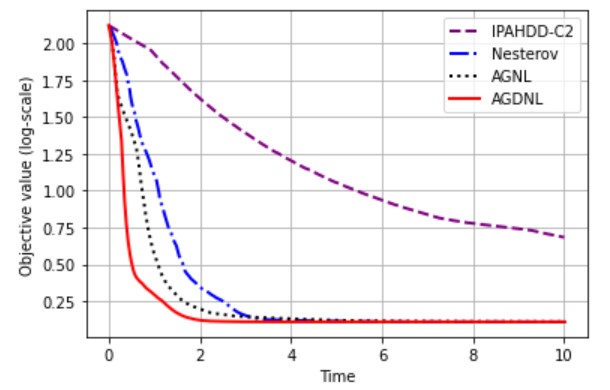}}
\subfloat[Minist dataset]{\label{fig:minist}\includegraphics[width=4cm]{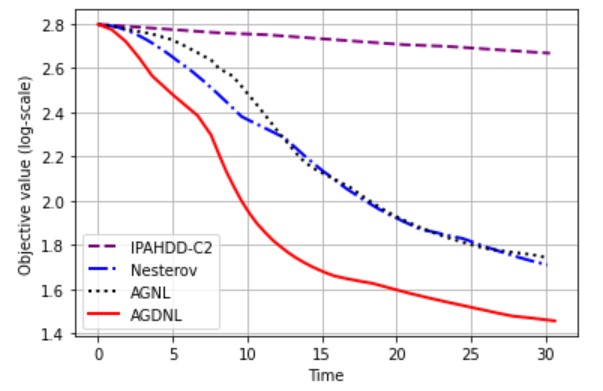}}
\caption{Results for the logistic regression problem.}
\label{fig:LR}
\end{figure}

We use six real-world datasets to compare AGDNL algorithm (Algorithm~\ref{algo:AGDNL}), AGNL algorithm (given in \eqref{AGNL}--\eqref{AGNL_ls}), Nesterov's accelerated gradient algorithm \cite{Nes83}, and IPAHDD-C2 algorithm which is the most efficient method in \cite{AAL21}. All datasets are obtained from the popular library for Support Vector Machines (LibSVM). We now recall IPAHDD-C2 algorithm
\begin{equation*}
\small
\begin{cases} 
z_n = x_n + \frac{1}{1 +\gamma h} (x_n -x_{n-1})\\ 
w_n \in \prox_{\frac{h}{1 +h\gamma}\phi}
 \left( \tfrac{1}{h(1+\gamma h)} (x_n -x_{n-1}) +\tfrac{\beta}{1+\gamma h} \nabla f(x_{n-1}) + \tfrac{\gamma \beta h}{1+\gamma h} \nabla f(x_n) - \frac{h}{1+\gamma h}\nabla f(z_n)\right)\\
x_{n+1}\!\!\!  = x_n -\beta h \nabla f(x_n) + hw_n.
\end{cases}
\end{equation*}
Set $\sigma = \frac{1}{1+\gamma h}$, $\lambda = \beta h$, $\alpha = \frac{1}{h(1+ \gamma h)}$, $\mu = h$, $\tau_n = \frac{h}{1+\gamma h}$, $\nu = \frac{h\beta \gamma}{1+\gamma h}$, and $\kappa = \frac{h}{1+\gamma h}$. Then IPAHDD-C2 algorithm becomes
\begin{equation*}
\begin{cases}
z_n &= x_n + \sigma (x_n - x_{n-1}),\\
w_n &\in \prox_{\tau_n \phi}\left(\alpha(x_n -x_{n-1}+\lambda\nabla f(x_{n-1} ))+\nu \nabla f(x_n) -\kappa \nabla f(z_n)\right ),\\
x_{n+1}\!\!\! &= x_n- \lambda\nabla f(x_n) +\mu w_n.
\end{cases}
\end{equation*}
According to Remark 2.2(iii), for $\phi =r\|\cdot\|$, we have $\prox_{\tau \phi}(x) = 0$ if and only if $\|x\| \leq r \tau$. Thus, when $\tau_n$ is set to be small, by the $w$-update, $w_n$ may not be zero when $\|\nabla f(u_n)\|$ and $\|x_n - x_{n-1}\|$ remain large, and so the dry-like friction function still impacts until $\|\nabla f(u_n)\|$ and $\|x_n - x_{n-1}\|$ are small enough. Therefore, for all algorithms, we choose $\Phi(x) =\|x\|$, $\theta_n = 10^{-8}$,  $r=1$, $\tau_n = 2 \times 10^{-8}$, $e_n =0$, and $\delta= 10^{-8}$.  Moreover, according to \eqref{eq:linesearch}, the smaller the value of $c$, the faster condition~\eqref{eq:linesearch} is satisfied, which consequently enhances the performance of AGDNL. Thus, in all experiments, $c$ is intentionally set to be small. Here, we set $c= 10^{-8}$. For IPAHDD-C2 algorithm, we choose $\lambda = \frac{0.4}{L}$, $\nu =\frac{0.5}{\mu L}$, $\kappa = \frac{1}{\mu L}$, and $\sigma =0.03$. We set the same maximum running time for all algorithms.

\begin{table}[!ht]
\centering
\caption{Numerical results of 4 comparative algorithms on real datasets}
\label{tab:results}
\begin{tabular}{l c c c l c c}
\hline
\multicolumn{1}{c}{\multirow{2}{*}{Data}} & Training & Testing & \multicolumn{1}{c}{\multirow{2}{*}{Class}} & \multicolumn{1}{c}{\multirow{2}{*}{Algorithm}} & Training & Testing \\
& data & data & & & accuracy & accuracy\\ \hline
\multirow{4}{*}{DNA} & \multirow{4}{*}{2000} & \multirow{4}{*}{1186} & \multirow{4}{*}{3} & IPAHDD-C2 & $97.75\% $ & $90.80 \%$\\
&  &  &   & Nesterov& $\bf 99.85\% $ & $\bf 94.01 \%$\\	
&  &  &   & AGNL& $\bf 99.85\% $ & $\bf 94.01 \%$\\	
&  &  &   &AGDNL&  $\bf 99.85\% $ & $\bf 94.01 \%$\\	 \hline
\multirow{4}{*}{Satimage} & \multirow{4}{*}{3104} & \multirow{4}{*}{1331} & \multirow{4}{*}{6} & IPAHDD-C2& $80.96\% $ & $80.00 \%$\\
&  &  &   & Nesterov &  $ 86.67\% $ & $\bf 83.6 \%$\\	
&  &  &   & AGNL& $86.50\% $ & $83.50 \%$\\	
&  &  &   & AGDNL& $\bf 86.78\% $ & $\bf 83.6 \%$\\	 \hline
\multirow{4}{*}{Pendigits} & \multirow{4}{*}{7494} & \multirow{4}{*}{3498} &\multirow{4}{*}{10} & IPAHDD-C2& $89.63\% $ & $84.07 \%$\\
&  &  &   & Nesterov&  $ 95.10\% $ & $ 89.87 \%$\\	
&  &  &   & AGNL& $ 94.99\% $ & $89.85 \%$\\	
&  &  &   & AGDNL& $\bf 95.20\% $ & $\bf 90.05 \%$\\	 \hline
\multirow{4}{*}{USPS} & \multirow{4}{*}{7291} & \multirow{4}{*}{2078} & \multirow{4}{*}{10} & IPAHDD-C2 &  $35.08\% $ & $32.93 \%$\\
&  &  &   & Nesterov&  $93.89\% $ & $88.24 \%$\\	
&  &  &   & AGNL& $94.32\% $ & $ 88.39 \%$\\	
&  &  &   & AGDNL&  $\bf 96.95\% $ & $\bf 90.23\%$	\\ \hline
\multirow{4}{*}{Letter} & \multirow{4}{*}{15000} & \multirow{4}{*}{5000} & \multirow{4}{*}{26}& IPAHDD-C2 &  $50.60\% $ & $49.36\%$\\
&  &  &   & Nesterov&  $76.30\% $ & $\bf 76.29\%$\\	
&  &  &   & AGNL& $76.44\% $ & $ 76.22 \%$\\	
&  &  &   & AGDNL&  $\bf 76.48\% $ & $ 76.27\%$	\\ \hline
\multirow{4}{*}{Minist} & \multirow{4}{*}{60000} & \multirow{4}{*}{10000} & \multirow{4}{*}{10}& IPAHDD-C2 &  $20.80\% $ & $20.78 \%$\\
&  &  &   & Nesterov&  $79.02\% $ & $79.0 \%$\\	
&  &  &   & AGNL& $77.96\% $ & $ 77.95 \%$\\	
&  &  &   & AGDNL&  $\bf 86.80\% $ & $\bf 86.80\%$	\\ \hline	\end{tabular}
\end{table}
Figure~\ref{fig:LR} shows that AGDNL takes less time to return a lower objective value for all datasets. We also observe from Table~\ref{tab:results} that after the same time, AGDNL algorithm and Nesterov's algorithm give better training accuracy and testing accuracy than IPAHDD-C2 on all datasets. Moreover, AGDNL algorithm is the most efficient method out of the four algorithms.

Still considering the multi-class logistic regression problem, we use three datasets (letter, seismic, usps) to compare two variants of our algorithm: AGDNL-M (monotone line search with $N=0$) and AGDNL (nonmonotone line search with $N=2$) with GIST algorithm in \cite{GZLHY13} which also uses nonmonotone line search and cGIPGM algorithm in \cite{WL19}. The parameters for GIST algorithm are chosen as in the \cite[Section 4.1]{GZLHY13} and for cGIPGM, we set $\beta := 0.98L/(L +\ell), \alpha := \alpha_{*}(\beta)$ and $\lambda := \min \{\alpha/\beta, (1.99 - 2\alpha)/(1 - \beta)\}/L$ as that in \cite{WL19}. Figure~\ref{fig:LR2} shows that AGDNL outperforms other algorithms. We also note that when $N=0$, the potential function $V$ in the criterion of our proposed algorithm is nonincreasing while in the case $N>0$, $V$ can be increasing at some points. However, $V$ is different from the objective function $F$, so in both cases, the objective function $F$ can also be increasing at some points but finally converges by Theorem~\ref{t:fullseq}, see Figure~\ref{fig:seismic_MN}. It follows from Figure~\ref{fig:LR2} that AGDNL has a faster convergence speed than AGDNL-M in most cases. Therefore, the nonmonotone line search criterion can further accelerate the convergence speed.  
\begin{figure}[!ht]
\centering
\subfloat[Letter dataset]{\label{fig:letter_MN}\includegraphics[width=4cm]{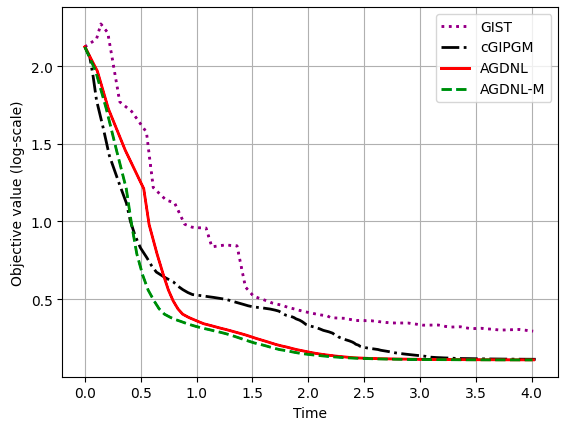}}
\subfloat[Seismic dataset]{\label{fig:seismic_MN}\includegraphics[width=4.2cm]{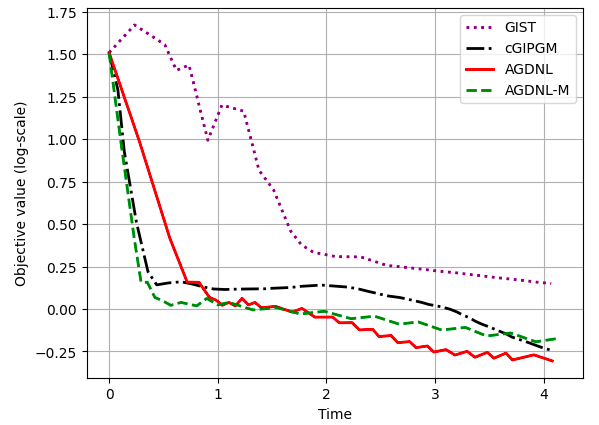}}
\subfloat[USPS dataset]{\label{fig:usps_MN}\includegraphics[width=4cm]{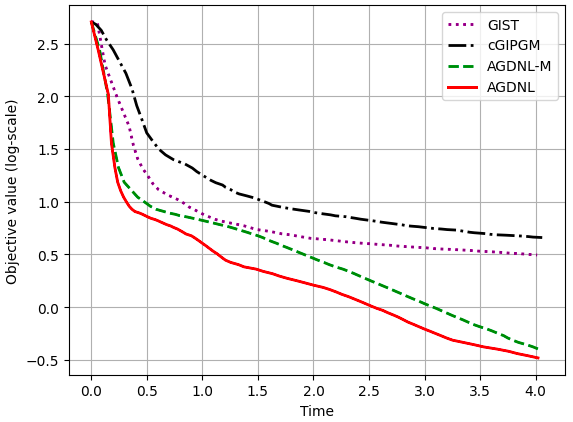}}
\caption{Logistic regression problem.}
\label{fig:LR2}
\end{figure}

We now present the performance of AGDNL with various choices of $\tau_n$. For all algorithms, we set $\tau_n = |\theta_n| + \delta$, where $\delta = 10^{-8}$. Additionally, we specify different values for $\theta_n$ in the following manner: $\theta_n = 10^{-8}$ for AGDNL-1, $\theta_n = 10^{-6}$ for AGDNL-2, $\theta_n = 10^{-4}$ for AGDNL-3, $\theta_n = 1/L_n -  10^{-8}$ for AGDNL-4, and $\theta_n = 2/L_n-10^{-8}$ for AGDNL-5. It is noteworthy that the parameter choice $\tau_n = 1/L_n $ noticeably enhances the performance of AGDNL, as illustrated in Figure~\ref{fig:LR3}.

\begin{figure}[!ht]
\centering
\subfloat[Letter dataset]{\label{fig:letter_tau}\includegraphics[width=4cm]{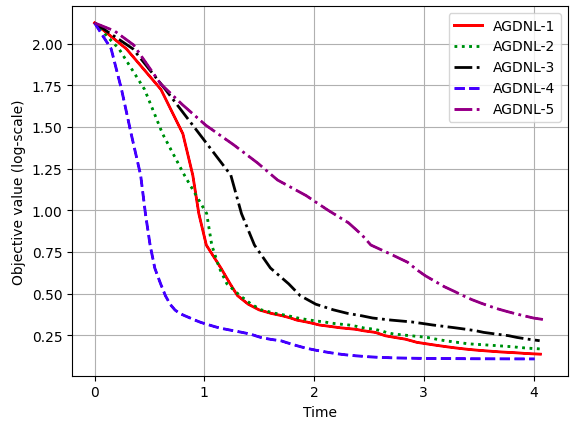}}
\subfloat[Pendigits dataset]{\label{fig:pendigits_tau}\includegraphics[width=4cm]{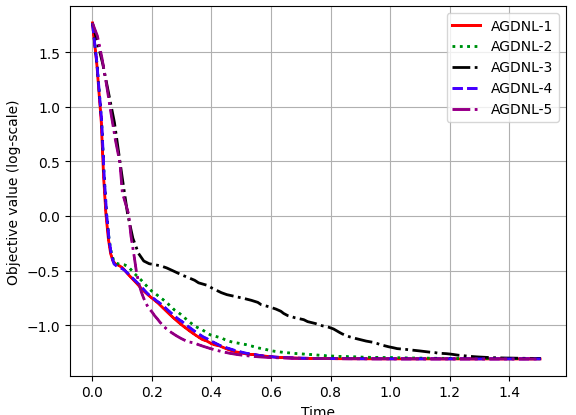}}
\subfloat[USPS dataset]{\label{fig:usps_tau}\includegraphics[width=4cm]{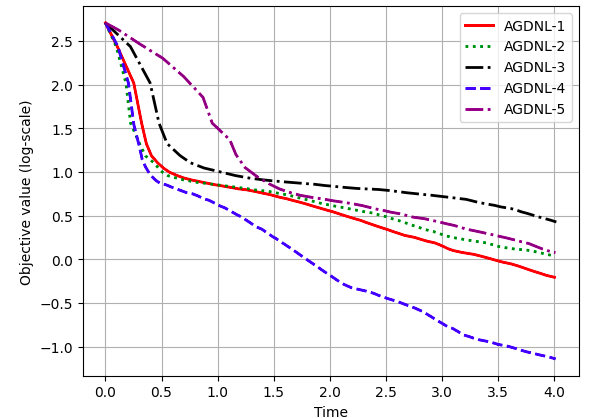}}
\caption{Logistic regression problem.}
\label{fig:LR3}
\end{figure}

\subsection{Matrix factorization in recommendation systems}

Recommendation systems have become increasingly popular and widely used in various industries and applications, including e-commerce, entertainment, and social media. In the context of video and music services such as Netflix, YouTube, and Spotify, recommendation systems analyze the user's watching or listening history and preferences to suggest new content that they are likely to enjoy. Similarly, in e-commerce platforms such as Amazon, these systems use past purchase history and user behaviour to suggest products that the user may be interested in buying. In these systems, matrix factorization is a powerful technique. 
 
Let $A$ be a feedback matrix in $\R^{m \times d}$, where $m$ is the number of users and $d$ is the number of items. We want to find a user matrix $U \in R^{m \times p}$, where row $i$ is the embedding for user $i$ and an item matrix $V \in \R^{d \times p}$, where row $j$ is the embedding for item $j$ such that the product $UV^T$
 is a good approximation of $A$. Therefore, we want to solve 
\begin{equation*}
\min_{X=[U,V]}f(X) = f(U, V ) = \frac{1}{2}\|A -UV^T\|^2_{F}   
\end{equation*}
for $U \in \R^{m \times p}$ and $V \in \R^{d \times p}$. This is a nonconvex problem. We used Movilens $100$K dataset which has $100,000$ ratings from $843$ users on $1682$ movies. Here, we choose several values of $p = 8, 10, 15, 20$. All algorithms have the same starting point that is chosen randomly. Figure \ref{fig:MF} shows the performance of five algorithms, which are IPAHDD-C2 algorithm, Nesterov's accelerated gradient method, AGDNL algorithm, IPAHDD-N-var algorithm \cite{AA20} and AGDNL-A algorithm which is the proposed algorithm where the limit point is an approached critical point. For all algorithms, the maximum running time is set as $T_{\max} =150 s$ in Figure \ref{Matrix_r8}, $T_{\max} =400 s$ in Figure \ref{Matrix_r10}, and $T_{\max} =1300 s$ in Figure \ref{Matrix_r15} and Figure \ref{Matrix_r20}. We see that AGDNL algorithm has a better performance than the other algorithms. It is also known that the classical framework of Nesterov’s accelerated gradient method is for convex functions possessing a Lipschitz continuous gradient.
Our proposed method does not require the convexity of the objective function. This example allows us to
investigate this potential advantage of our method.

\begin{figure}[!ht]
\centering
\subfloat[$p=8$]{\label{Matrix_r8}\includegraphics[width=4.5cm]{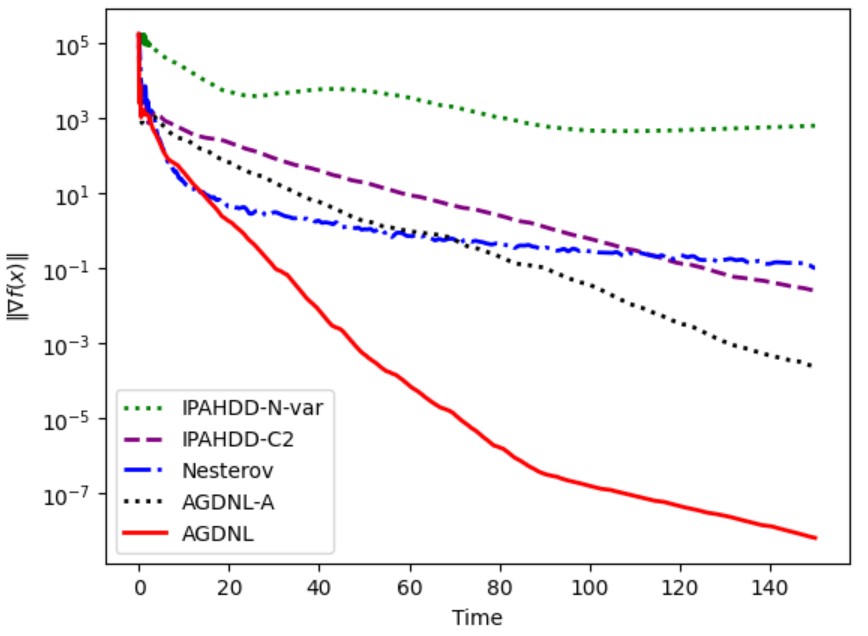}}
\subfloat[$p=10$]{\label{Matrix_r10}\includegraphics[width=4.5cm]{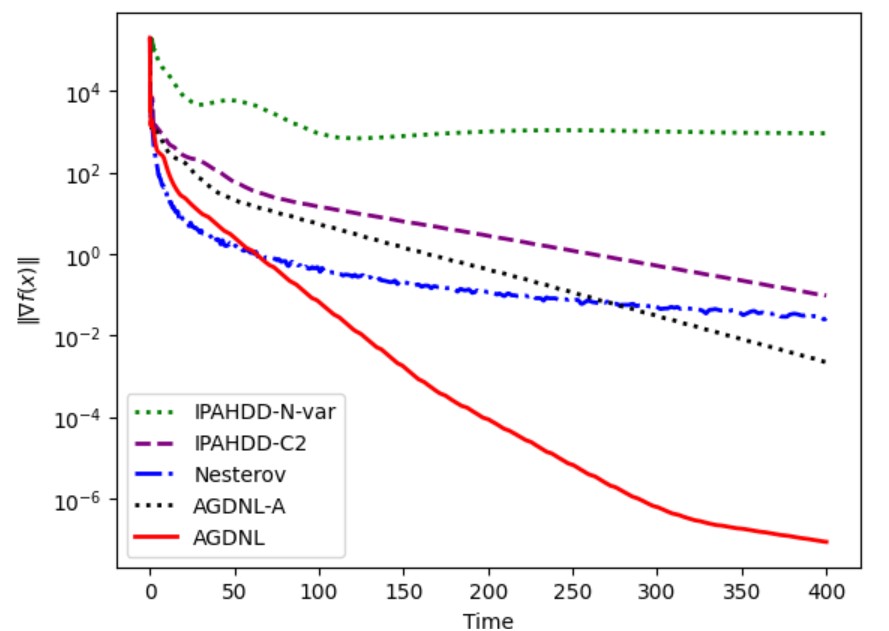}}

\subfloat[$p=15$]{\label{Matrix_r15}\includegraphics[width=4.5cm]{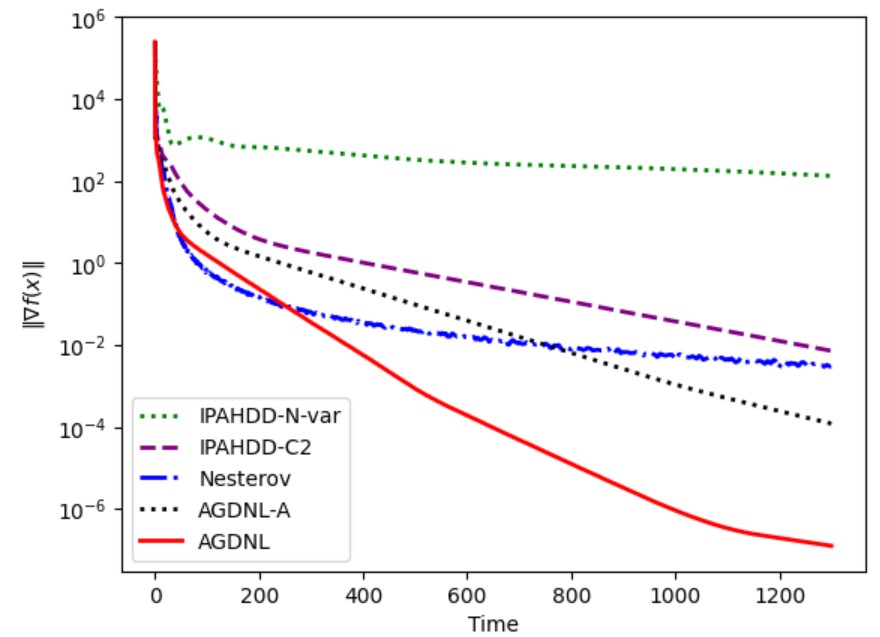}}
\subfloat[$p=20$]{\label{Matrix_r20}\includegraphics[width=4.5cm]{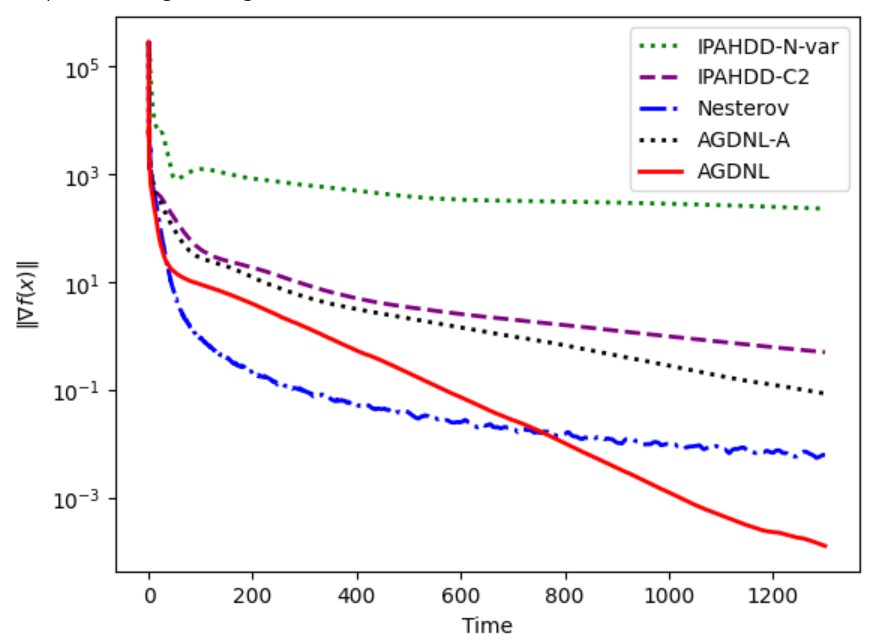}}
\caption{Results for matrix factorization.}
\label{fig:MF}
\end{figure}

\subsection{Ridge regression problem}

Consider the ridge regression problem
\begin{equation*} 
\min_{x \in \R^{d}} f(x):=\tfrac{1}{2}\|Ax-b\|^2 + \lambda\|x\|^2,
\end{equation*}
where $A \in \R^{m \times d}$ is a random sparse matrix and $b \in \R^{m}$ is randomly chosen. Here, in Figure \ref{fig:R}, $A \in \R^{1200 \times 5000}$, $b \in \R^{1200}$, and $\lambda =0.01$, while in Figure \ref{fig:RE}, $A \in \R^{1000 \times 10000}$, $b \in \R^{1000}$. AGDNL-N is AGDNL algorithm using nonconvex dry-like friction $\phi(x) = \max\{\frac{\alpha}{2}\|x\|^2 +r\|x\|, -\frac{\alpha}{2}\|x\|^2+2r\|x\|\}$ whose proximal operator is given in Remark~\ref{r:DLF}\ref{r:DLF_weaklycvxfun}. We choose $\alpha =0.01$ and $r=1$. Meanwhile, AGDNL-E is AGDNL algorithm using the error term $e_i$. In Figure~\ref{fig:R} and Figure~\ref{fig:RE}, we set $e_n = \frac{1}{n L} \frac{\xi}{\|\xi\|}$, where $\xi$ is a random vector in $\R^d$ with the uniform distribution on $(0, 1)^d$, $\tau_n =10^{-8}$, while in Figure~\ref{fig:E}, $e_n = \left(\frac{1}{L_n}-10^{-8}\right)\frac{\nabla f(x_n}{\|\nabla f(x_n)\|}$ and $\tau_n =\frac{1}{L_n}$.  All algorithms have the same initial point. In general, AGDNL algorithm has similar performance to AGDNL-E algorithm and AGDNL-N algorithm. However, there are instances, as illustrated in Figure \ref{fig:RR}, where the AGDNL-N and AGDNL-E demonstrate slightly superior performance compared to AGDNL and Nesterov's accelerated algorithm. Moreover, it also shows the robustness of our proposed algorithm.

\begin{figure}[!ht]
\centering
\subfloat[$A \in \R^{1200 \times 5000}$]{\label{fig:R}\includegraphics[width=3.5cm]{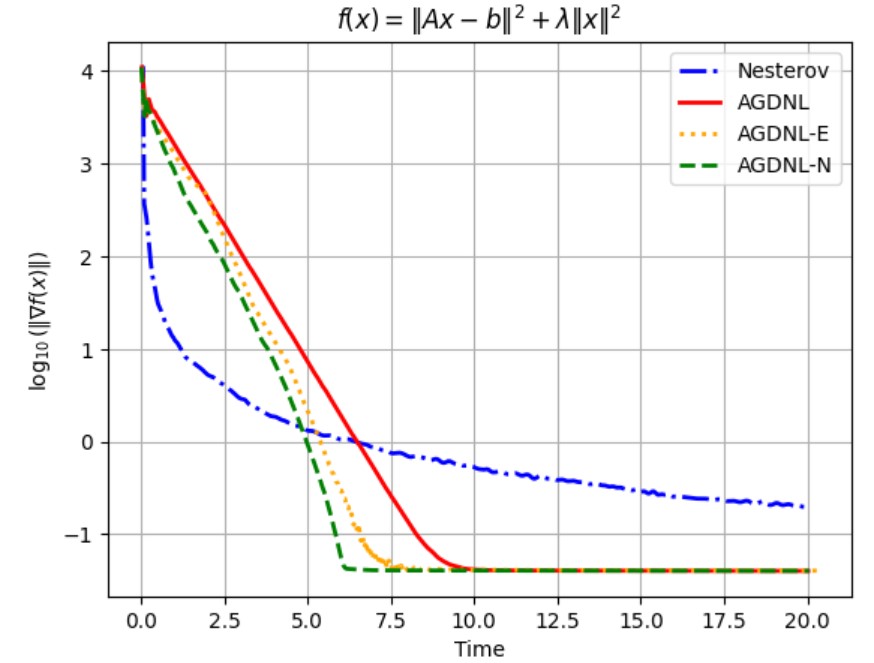}}
\subfloat[$A \in \R^{1000 \times 10000}$]{\label{fig:RE}\includegraphics[width=4cm]{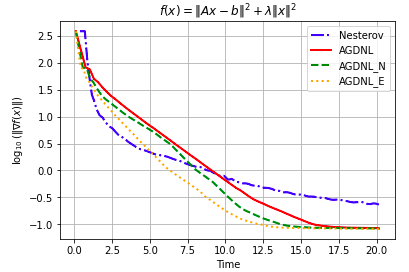}}
\subfloat[$A \in \R^{2000 \times 7000}$]{\label{fig:E}\includegraphics[width=3.5cm]{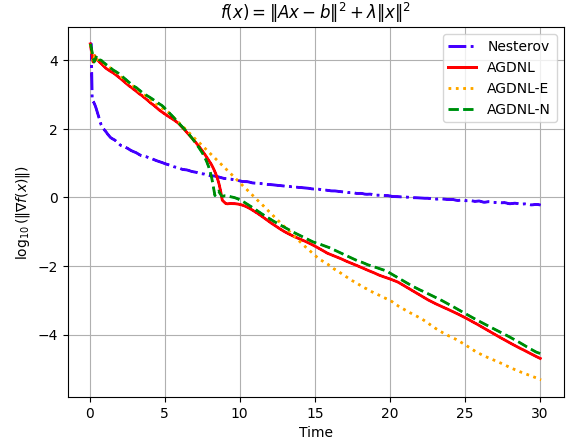}}
\caption{Ridge regression problem}
\label{fig:RR}
\end{figure}

\subsection{Quadratic problem with Dolan--Mor\'e performance profile}

We consider the quadratic problem $\min_{x \in \R^d}f(x) =\|Ax-b\|^2$, where $A \in \R^{m \times d}$ ($m \leq d$) and $b \in \R^m$ are chosen randomly. To compare the algorithms, we use the performance profiles developed by Dolan and Mor\'e \cite{DM02}. Let $P$ be a set of $50$ different problems with $50$ sparse matrices $A$ which are generated randomly in Python with size ranging from $m = 24$ to $m = 1300$ and from $d = 1300$ to $d = 20000$. Let $S$ be the set of the four solvers that are IPAHDD-C2,  Nesterov's accelerated gradient algorithm, AGDNL and AGDNL-N. For AGDNL algorithm, we use convex dry friction $\phi(x) = r\|x\|$ (see Remark~\ref{r:DLF}\ref{r:DLF_cvx}), while for AGDNL-N, we use nonconvex dry-like friction $\phi(x) = \max\{\frac{\alpha}{2}\|x\|^2 +r\|x\|, -\frac{\alpha}{2}\|x\|^2+2r\|x\|\}$ (see Remark~\ref{r:DLF}\ref{r:DLF_weaklycvxfun}). For each $p \in P$ and $s \in S$, the performance ratio is defined by
 \begin{equation*}
r_{p,s} = \frac{t_{p,s}}{\min\{t_{p,s} : s\in S\}},
 \end{equation*}
 where $t_{p,s}$ is the computing time for solver $s$ to solve problem $p$. The performance of the solver $s \in S$ is given by $\rho_s(t) = \frac{1}{n_p} \text{size}\{p \in P : r_{p,s} \leq  t\}$,
where $n_p$ is the number of problems. Thus, $\rho_s(t)$ is the probability for solver $s \in S$ that a performance ratio $r_{p,s}$ is within a factor $t \in \R$ of the best possible ratio. For all algorithms, we choose the same initial points and the same stopping criterion, i.e.,
either the computing time exceeds 1000s or $\|\nabla f(x)\| \leq 0.1$.
It can be observed from Figure \ref{fig:Q} that AGDNL and AGDNL-N are the most efficient algorithms, followed by Nesterov's accelerated gradient method.
\begin{figure}
\centering
\includegraphics[width=4.5cm]{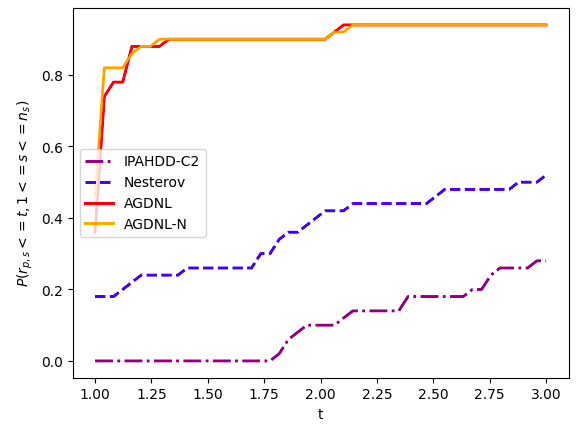} 
\caption{Quadratic problems}
\label{fig:Q}
\end{figure}

\section{Conclusion}
\label{s:conclusion}

We have proposed a fast gradient algorithm with dry-like friction and nonmonotone line search (AGDNL), which is more general and more efficient than inertial gradient algorithms with Hessian damping and dry friction in the literature. The iterative sequence generated by AGDNL algorithm is bounded (with all strong cluster points being critical points) or strongly convergent to an ``approximate'' critical point of the objective function. The convergence of the whole sequence to a critical point is also guaranteed in both smooth and nonsmooth settings under KL assumptions. We have derived the KL exponent for the Moreau envelope of KL functions in Hilbert spaces that are not necessarily convex nor continuous. The efficiency of the proposed algorithm has been illustrated by numerical simulations.

\appendix
\section{Proof of a closed form}
\label{s:appendix}

\begin{proof}[Proof of the closed form in Remark~\ref{r:DLF}\ref{r:DLF_weaklycvxfun}]
By Cauchy--Schwarz inequality, for all $x, y\in \mathcal{H}$ and all $\tau \in \mathbb{R}_{++}$,
\begin{align} 
\phi(y) + \frac{1}{2 \tau}\|x-y\|^2 &\geq \phi(y) + \frac{1}{2 \tau}\|x\|^2 - \frac{1}{\tau}\|x\|\|y\| + \frac{1}{2 \tau} \|y\|^2 \notag\\
&=\begin{cases}
\varphi_1(\|y\|) &\text{if~} \|y\| >\frac{r}{\alpha},\\
\varphi_2(\|y\|) &\text{if~} \|y\| \leq \frac{r}{\alpha},
\end{cases}
\label{eq:M}
\end{align}
where $\varphi_1(t):= \frac{1}{2}(\alpha + \frac{1}{\tau})t^2 + (r -\frac{1}{\tau}\|x\|)t+ \frac{1}{2\tau}\|x\|^2$ and $\varphi_2(t): = \frac{1}{2}(-\alpha + \frac{1}{\tau})t^2 + (2r -\frac{1}{\tau}\|x\|)t+ \frac{1}{2\tau}\|x\|^2$.  Let $\tau \in (0,\frac{1}{\alpha})$. Then $\alpha + \frac{1}{\tau} > -\alpha + \frac{1}{\tau} >0$. It can be seen that 
\begin{equation}\label{e:Phi1}
\text{$\varphi_1$ is decreasing on $\left(-\infty, \frac{\|x\| -r\tau}{ 1 +\alpha \tau}\right]$ and increasing on $\left[\frac{\|x\| -r\tau}{ 1 +\alpha \tau}, +\infty\right)$} 
\end{equation}
and that
\begin{equation}\label{e:Phi2}
\text{$\varphi_2$ is decreasing on $\left(-\infty, \frac{\|x\| -2r\tau}{ 1 -\alpha \tau}\right]$ and increasing on $\left[\frac{\|x\| -2r\tau}{ 1 -\alpha \tau}, +\infty\right)$}.    
\end{equation}
We now distinguish the following cases.

\emph{Case 1:} $\|x\| \leq 2r\tau$. In view of \eqref{e:Phi2}, $\min_{\|y\| \leq \frac{r}{\alpha}} \varphi_2(\|y\|) =\varphi_2(0) =\frac{1}{2 \tau}\|x\|^2$. Since $\alpha < \frac{1}{\tau}$ and $\frac{1}{\tau}\|x\| \leq 2r$, we have that, for all $\|y\| \geq \frac{r}{\alpha}$,
\begin{equation*}
\varphi_1(\|y\|) > \alpha \|y\|^2 - r\|y\| + \frac{1}{2\tau}\|x\|^2 = \|y\| (\alpha \|y\| -r) + \frac{1}{2\tau}\|x\|^2 \geq \frac{1}{2\tau}\|x\|^2.
\end{equation*}
By combining with \eqref{eq:M}, we derive that, for all $y\in \mathcal{H}$, $\phi(y) + \frac{1}{2 \tau}\|x-y\|^2 \geq \frac{1}{2\tau}\|x\|^2$, where the equality holds if and only if $y =0$. This means $\prox_{\tau \phi}(x) =0$.

\emph{Case 2:} $2 r \tau < \|x\| \leq r \tau +\frac{r}{\alpha}$. Then $\frac{\|x\| -r\tau}{ 1 +\alpha \tau} \leq \frac{r}{\alpha (1 +\alpha \tau)} <\frac{r}{\alpha}$ and  $\frac{\|x\| - 2 r\tau}{1 -\alpha \tau} \leq \frac{\frac{r}{\alpha} - r\tau}{ 1 -\alpha \tau} =\frac{r}{\alpha}$. It follows from \eqref{e:Phi1} and \eqref{e:Phi2} that $\varphi_1$ is increasing on $[\frac{r}{\alpha}, +\infty)$ and $\min_{\|y\| \leq \frac{r}{\alpha}} \varphi_2(\|y\|) =\varphi_2\left(\frac{\|x\| -2 r\tau}{1 -\alpha \tau}\right)$. Therefore, $\min_{\|y\| >\frac{r}{\alpha}} \varphi_1(\|y\|) > \varphi_1\left(\frac{r}{\alpha}\right) = \varphi_2\left(\frac{r}{\alpha}\right) \geq \varphi_2\left(\frac{\|x\| -2 r\tau}{1 -\alpha \tau}\right),$ and so, for all $y\in \mathcal{H}$, $\phi(y) + \frac{1}{2\tau}\|x-y\|^2 \geq \varphi_2\left(\frac{\|x\|-2r\tau}{1 -\alpha \tau}\right),$ where the equality holds if and only if $\langle x, y\rangle =\|x\|\|y\|$ and $\|y\| =\frac{\|x\|-2r\tau}{1 -\alpha \tau}$. As a result, $\prox_{\tau \phi}(x) =\frac{\|x\| -2 r\tau}{1 -\alpha \tau} \frac{x}{\|x\|}$.

\emph{Case 3:} $r\tau + \frac{r}{\alpha} < \|x\| \leq 2 r\tau + \frac{r}{\alpha}$. Then $\frac{\|x\| -r\tau}{ 1 +\alpha \tau} \leq \frac{r \tau +\frac{r}{\alpha}}{1 +\alpha \tau} =\frac{r}{\alpha}$ and $\frac{\|x\| - 2 r\tau}{1 -\alpha \tau} \geq \frac{\frac{r}{\alpha} - r\tau}{ 1 -\alpha \tau} =\frac{r}{\alpha}$. We derive from \eqref{e:Phi1} and \eqref{e:Phi2} that $\min_{\|y\| >\frac{r}{\alpha}} \varphi_1(\|y\|) > \varphi_1\left(\frac{r}{\alpha}\right) = \varphi_2\left(\frac{r}{\alpha}\right) \text{~~and~~}
\min_{\|y\| \leq \frac{r}{\alpha}}\varphi_2(\|y\|) = \varphi_2\left(\frac{r}{\alpha}\right).$ This together with \eqref{eq:M} implies that, for all $y\in \mathcal{H}$, $\phi(y) + \frac{1}{2\tau}\|x-y\|^2) \geq \varphi_2\left(\frac{r}{\alpha}\right)$, where the equality holds if and only if $\langle x, y\rangle =\|x\|\|y\|$ and $\| y\| = \frac{r}{\alpha}$. Therefore, $\prox_{\tau \phi}(x) = \frac{r}{\alpha} \frac{x}{\|x\|}$.

\emph{Case 4:} $\|x\| > 2 r\tau +\frac{r}{\alpha}$. Then $\frac{\|x\| -r\tau}{1 + \alpha \tau} > \frac{r}{\alpha}$ and $\frac{\|x\| - 2 r\tau}{1 -\alpha \tau} > \frac{r}{\alpha}$. In view of \eqref{e:Phi1} and \eqref{e:Phi2}, $\min_{\|y\| > \frac{r}{\alpha}} \varphi_1(\|y\|) = \varphi_1\left(\frac{\|x\| -r\tau}{1 + \alpha \tau}\right)$ and $\min_{\|y\| \leq \frac{r}{\alpha}} \varphi_2(\|y\|) = \varphi_2\left(\frac{r}{\alpha}\right) = \varphi_1\left(\frac{r}{\alpha}\right) \geq \varphi_1\left(\frac{\|x\| -r\tau}{1 + \alpha \tau}\right),$ which implies that $\prox_{\tau \phi}(x) = \frac{\|x\| - r\tau}{1+\alpha \tau} \frac{x}{\|x\|}$.
\end{proof}

\bibliographystyle{siamplain}
\bibliography{references}

\end{document}